\documentclass[12pt, a4paper]{article}
\usepackage{amsmath,amsthm,amsfonts,amssymb,verbatim,afterpage}
\usepackage[numbers]{natbib}

\usepackage{tikz}
\usetikzlibrary{arrows.meta, positioning, shadows, shapes, calc}
\usepackage{graphicx}

\theoremstyle{plain}
\newtheorem{theorem}{Theorem}

\newtheorem{lemma}{Lemma}
\newtheorem{proposition}{Proposition}
\theoremstyle{remark}
\newtheorem{definition}{Definition}
\newtheorem{assumption}{Assumption}
\theoremstyle{remark}
\newtheorem{remark}{Remark}

\usepackage{enumitem}

\usepackage{mathtools}
\DeclarePairedDelimiter\abs{\lvert}{\rvert}
\DeclarePairedDelimiter\norm{\lVert}{\rVert}

\makeatletter
\def\ow{\overline{w}}
\let\oldabs\abs
\def\abs{\@ifstar{\oldabs}{\oldabs*}}

\def\pd{\partial}

\def\Nb{\overline{\mathbb{N}}}

\let\oldnorm\norm
\def\norm{\@ifstar{\oldnorm}{\oldnorm*}}
\makeatother

\usepackage{mathrsfs}
\usepackage{hyperref}

\newcommand{\R}{\mathbb{R}}

\newcommand{\E}{\mathbb{E}}
\newcommand{\PP}{\mathbb{P}}

\usepackage{bbm}
\newcommand{\1}{\mathbbm{1}}

\DeclareMathOperator{\supp}{\mathrm{supp}}

\def\cH{\mathcal{H}}

\newcommand{\Env}{\mathcal{E}}

\title{Iterated harmonic envelopes for multidimensional optimal stopping}

\author{John Moriarty\thanks{School of Mathematical Sciences, Queen Mary University of London.\newline
Email: \texttt{j.moriarty@qmul.ac.uk}}}
\date{25 November 2026}
\begin{document}
\usetikzlibrary{decorations.pathreplacing}
\maketitle

\begin{abstract}
Motivated by the contrast between the classical least-con\-cave-ma\-jor\-ant representation of Dynkin and Yushkevich in one dimension and the substantially more complicated potential-theoretic structure of optimal stopping in higher dimensions, we identify a fundamental barrier: a localisation obstruction caused by the geometry of Brownian paths in dimensions $d\ge2$. We overcome it with an iterative construction based on harmonic-function envelopes, each an infimum of harmonic replacements on subdomains of the current non-contact set, with the previous envelope as boundary data. From a continuous gain function $g$, this produces a decreasing sequence of upper bounds $w_n$ converging uniformly to the value function $V$ of optimal stopping. The process may be Brownian motion or, more generally, any diffusion with a short list of properties, including polarity of points. For Brownian motion, every iterate inherits the modulus of continuity of the data. Exponential discounting may be included, and the state space may be a bounded regular domain with absorption at the boundary, or the whole space under a condition at infinity. Equivalently, the iterates converge to the solution of the associated obstacle problem. At each iteration an explicit stopping rule gives a lower bound for $V$. Both bounds are computed from the construction alone, without knowledge of $V$, and their gap falls below any given threshold after finitely many iterations. The first exit time from the limit non-contact set is optimal.
\end{abstract}

\medskip\noindent\textbf{MSC 2020:} 60G40 (primary); 31B05, 60J45, 60J65 (secondary).

\noindent\textbf{Keywords:} optimal stopping; superharmonic majorant; harmonic replacement; Perron's method; polarity of points; Brownian motion.

\section{Introduction}
\label{sec:statement}

Superharmonic functions admit both analytic and probabilistic representations. While the classical PDE approach emphasises domain-wide regularity, the probabilistic viewpoint offers localised characterisations via stopping times. In this paper we use localisation to approximate arbitrarily closely the value of the multidimensional optimal stopping problem
\begin{align}\label{eq:defv}
V(x) := \sup_{\tau \in \mathscr{T}} \E[g(X^x_\tau)],
\end{align}
where $X^x$ is multidimensional Brownian motion started at $x$, or more generally a diffusion with the properties (P1)--(P7) of Section~\ref{sec:dpt}, $\mathscr T$ is the class of admissible stopping times, and the gain $g$ is continuous. These properties are classical for Brownian motion, and a reader interested only in Brownian motion may read the whole paper in that case. Exponential discounting may be included, and the state space may be a bounded regular domain with absorption at the boundary or, under a condition at infinity, the whole space.

The key idea is to take harmonic-function envelopes only on the current non-contact set and then update this set, producing a decreasing sequence of domains. This localisation overcomes the obstruction caused by the fact that Brownian motion in dimension $d\ge2$ is unlikely to hit a set of small capacity (Newtonian for $d\ge3$, logarithmic for $d=2$).

In one dimension, Dynkin and Yushkevich \cite{Dynkin_Yushkevich} reduced the problem to constructing the smallest concave majorant, yielding a fully explicit representation. In dimensions $d \ge 2$, however, the localisation obstruction means that $g$ can be large on a set that contributes little to $V$ (Section~\ref{sec:sball}). A harmonic function subject to a global majorisation constraint must nevertheless exceed $g$ there, and by Harnack's inequality it then overestimates the value function by an amount which does not vanish with the capacity. Accordingly, to our knowledge, no construction of Dynkin--Yushkevich type has been available.

Our main result, Theorem~\ref{thm:main}, shows that $w_n\downarrow V$ uniformly on $\overline\Lambda$. Here, $(w_n)_{n \ge 1}$ is the decreasing sequence of harmonic-function envelopes, which take the previous envelope as boundary data. The optimal stopping region is then the contact set of $g$ with the limit envelope, and the first exit time from the limit non-contact set is optimal. After finitely many steps the construction yields explicit near-optimal stopping rules, together with a certificate of their accuracy, both computed from the construction alone, without knowledge of $V$ (Theorem~\ref{thm:main}(iii)). This yields a multidimensional, iterative analogue of the Dynkin--Yushkevich representation.

A probabilistic approach also overcomes a further challenge in $d \ge 2$, namely the potential lack of regularity of free boundaries. A comparison principle (Lemma~\ref{lem:comparison}) shows that harmonic functions which majorise $g$ on an open set, and dominate the previous envelope where paths leave it, yield admissible upper bounds. As in Perron's method, this establishes that the limit is harmonic off the contact set, by considering Poisson modifications of the limit envelope applied on small balls. Applied to an admissible set enlarged by a small ball around a point of the non-contact set, at the cost of a small margin above $g$, the comparison principle shows that every envelope is continuous (Proposition~\ref{prop:contenv}). Apart from general properties of the process, this uses only polarity of points; it keeps the recursion within continuous superharmonic majorants and, with a lifted form of the Poisson modification, gives a continuous limit. For Brownian motion, applied on translated open sets of arbitrary regularity (Lemma~\ref{lem:translation}), the comparison principle shows moreover that every envelope inherits the modulus of continuity of the data, uniformly in $n$.

The paper is organised as follows. Section~\ref{sec:setting} introduces the process through the properties (P1)--(P7), which are classical for Brownian motion, together with the spatial settings, the construction and the motivating example. Section~\ref{sec:branching} states and proves the main results in this generality. Section~\ref{sec:translation} shows that, for Brownian motion, every envelope inherits the modulus of continuity of the data. Appendix~\ref{app:conseq} contains proofs of standard facts, and Appendix~\ref{app:ceilings} constructs ceilings.

\subsection{Background and context}
\label{sec:literature}

\paragraph{The Perron method and balayage.}
Classically, the value function of an optimal stopping problem is the least superharmonic majorant (LSM) of the gain \cite{Peskir2006} and, in analytic terms, the solution of the associated obstacle problem \cite{bensoussan1982stochastic}. The Perron envelope constructs the LSM as a pointwise infimum over superharmonic
majorants:
\[
w(x) = \inf\bigl\{h(x) : h \text{ superharmonic on } \Lambda,\ h \geq g\bigr\}.
\]
This provides an existence-based construction: the object is well defined without any prior assumption of existence. It does not, however, supply a canonical approximating sequence. Poincar\'{e}'s balayage procedure is more constructive in spirit: one repeatedly replaces the values of a superharmonic majorant on balls by their Poisson integrals, along a sequence of balls chosen in advance.

\paragraph{Wiener exhaustion.}
Wiener's solution to the Dirichlet problem via exhaustion by regular subdomains is constructive in an approximation-theoretic sense: the problem is solved on a monotone increasing sequence of regular subdomains, and the resulting harmonic functions converge to the Perron solution. The approximation proceeds from below rather than above.

\paragraph{Optimal stopping in one and higher dimensions.}
For regular diffusions in one dimension, Dayanik and Karatzas
\cite{Dayanik_Karatzas} characterise the value function via concave majorants,
with extensions to nonlinear settings in \cite{km}. In higher dimensions,
analytic approaches based on variational inequalities, PDE theory, and viscosity
solutions provide general existence and regularity results
\cite{bensoussan1982stochastic, Reikvam01011998, peskir2019continuity}.
Additional structure such as monotonicity or symmetry permits explicit solutions
in special cases \cite{Peskir2006, 10.1214/13-AAP956}. Christensen, Crocce, Mordecki and Salminen \cite{CHRISTENSEN20192561} represent value functions of multidimensional diffusions through Green kernels and Martin boundaries, characterising stopping regions via integral equations. In mathematical finance, the problem of \emph{representability} relates to the question of whether the value function can be constructed in a single step without iteration of envelopes, in one time and one space dimension \cite{christensen2022american}.
Deep learning methods have been applied to obtain numerical approximations in
high-dimensional settings \cite{becker2019deep, lauriere2025deep}. None of
these approaches yields a constructive representation of the LSM in terms of
tractable building blocks with a natural sequential approximation.

\paragraph{Numerical methods and two-sided bounds.}
Discretisations of the obstacle problem lead to linear complementarity
problems, for which standard solvers include projected successive
over-relaxation \cite{Cryer1971}, policy iteration in the form of Howard's
algorithm \cite{Howard1960,BokanowskiMarosoZidani2009}, penalty methods
\cite{ForsythVetzal2002} and primal--dual active set strategies
\cite{HintermullerItoKunisch2003}; see also \cite{GlowinskiLionsTremolieres}.
These methods operate on the discrete problem, and their iterates are not in
general majorants of the continuous solution. Policy iteration is the closest
in spirit to the present construction: it alternates between solving linear
problems on the current continuation set and updating that set, and it
terminates after finitely many steps. In mathematical finance, methods for
optimal stopping are commonly assessed by two-sided bounds. The value of an
explicit stopping rule is a lower bound, and the dual representation of Rogers
\cite{Rogers2002} and Haugh and Kogan \cite{HaughKogan2004} gives an upper
bound for each choice of martingale, which is tight at the martingale part of
the Snell envelope; it applies to essentially arbitrary adapted payoff
processes, including finite horizons and general underlying dynamics. Andersen
and Broadie \cite{AndersenBroadie2004} construct the martingale from an
approximate stopping rule, and the learned rules of \cite{becker2019deep} are
certified in the same way. Since these bounds are estimated by simulation from
approximate continuation values, they hold up to statistical error, and their
gap need not vanish. Our construction is defined at the continuous level and,
within its more restricted class of problems, gives bounds of the same form
(Theorem~\ref{thm:main}(iii)): each iterate $w_n$ is a continuous superharmonic
majorant of the solution, the iterates converge uniformly
(Theorem~\ref{thm:main}), the value of an explicit stopping rule is a
lower bound, both bounds are deterministic, and for each $\varepsilon>0$ the
gap falls below $2\varepsilon$ after finitely many steps.

\paragraph{The present paper.}
The construction introduced here can be viewed as a structured refinement of
the Perron framework. Rather than taking the infimum over all superharmonic
majorants, each step takes the infimum over harmonic functions on subdomains of the current non-contact set, with boundary values given by the previous envelope, and the procedure is then iterated. Poisson modifications on balls, lifted slightly where necessary, are among the admissible steps (Lemmas~\ref{lem:comparison} and~\ref{lem:lifted}), but they are selected by the construction rather than prescribed in advance. This yields a monotone sequence of envelopes converging uniformly to the value function.

\section{The framework}\label{sec:setting}

\subsection{The process and the two spatial settings}\label{sec:process}

Throughout, $d\ge2$, and $\beta\ge0$ is a discount rate. Let $\Omega=C(\R_{\ge0};\R^d)$ with coordinate process $Y=(Y_t)_{t\ge0}$, and let $(\PP^z)_{z\in\R^d}$ be a family of probability measures on $\Omega$ with $\PP^z(Y_0=z)=1$, such that $z\mapsto\PP^z(A)$ is Borel for each measurable $A$. We write $\E^z$ for expectation, $(\mathcal F_t)$ for the usual augmentation (completion by the $\PP^z$-null sets, made right-continuous) of the filtration generated by $Y$, and $P_t(z,\cdot):=\PP^z(Y_t\in\cdot)$ for the transition kernels. The main example is Brownian motion, for which $\PP^z$ is the law of $z+B$, with $B$ a standard Brownian motion.

For each $A \subset \R^d$, write $\overline{A}$ for its closure, $A^c$ for its complement, $\pd A = \overline{A} \cap \overline{A^c}$ for its boundary, and $A+v:=\{a+v:a\in A\}$ for its translate by $v\in\R^d$. Write $B(x;r)$ and $B^\circ(x;r)$ for the closed and open balls of radius $r$ centred at $x$, respectively.

For a Borel set $A\subset\R^d$ and an open set $O\subset\R^d$ let
\[
\tau_A:=\inf\{t\ge0:Y_t\in A\},\qquad\tau^+_O:=\inf\{t>0:Y_t\notin O\},
\]
with the convention $\inf\emptyset:=\infty$ here and throughout,
so that $\tau^+_O$ is the first exit from $O$ after time $0$. These are stopping times, since the filtration satisfies the usual conditions. The strict inequality matters only at points of $\pd O$ from which the path enters $O$ immediately. A point $y\in\pd O$ is \emph{regular} for $O$ if $\PP^y(\tau^+_O=0)=1$, that is, if the path started at $y$ leaves $O$ immediately. For $U$ bounded and open and $F:\R^d\to\R$ Borel and bounded on bounded sets, let
\[
h^U_F(z):=\E^z\bigl[e^{-\beta\tau^+_U}F(Y_{\tau^+_U})\bigr],\qquad z\in\R^d.
\]
The exit time $\tau^+_U$ is finite almost surely by \ref{P:small} below. A Borel function $u$, bounded on bounded sets, is \emph{$\beta$-harmonic} in an open set $O$ if $u=h^G_u$ on $G$ for every open set $G$ whose closure is a compact subset of $O$.

We consider two settings.
\begin{enumerate}[label=(\alph*)]
\item \emph{A bounded regular domain.} Let $\Lambda\subset\R^d$ be a bounded open set all of whose boundary points are regular for $\Lambda$. For Brownian motion the main example is the open unit ball; regularity holds more generally under the exterior cone condition, and it cannot be dropped (Remark~\ref{rem:regular} in Appendix~\ref{app:ceilings}). For $x\in\overline\Lambda$, let $X^x$ be $Y$ under $\PP^x$, stopped at $\tau_{\pd\Lambda}$, so that $X^x$ is absorbed on $\pd\Lambda$. The admissible stopping times are those satisfying $\tau\le\tau_{\pd\Lambda}$, so that stopping is enforced at absorption.
\item \emph{The whole space.} Let $\Lambda:=\R^d$, so that $\pd\Lambda=\emptyset$ and $\tau_{\pd\Lambda}=\infty$; let $X^x$ be $Y$ under $\PP^x$. All stopping times, finite or not, are admissible.
\end{enumerate}
\begin{assumption}\label{ass:breg}
Either we are in setting (a) and $g\in C(\overline\Lambda)$; or we are in setting (b), and $g\ge0$ is continuous and vanishes at infinity.
\end{assumption}
The optimal stopping value is
\begin{equation}\label{eq:defV}
V(x):=\sup_{\tau\in\mathscr T}\E\bigl[e^{-\beta\tau}g(X^x_\tau)\1_{\{\tau<\infty\}}\bigr],\qquad x\in\overline\Lambda,
\end{equation}
where $\mathscr T$ is the class of admissible stopping times. When the process is written $X^x$, we write $\E$ for $\E^x$. In setting (a) the indicator is redundant and $V=g$ on $\pd\Lambda$; no sign or support condition is imposed on $g$. In setting (b) never stopping yields zero, and the sign condition is no restriction: declining to stop on $\{g(X^x_\tau)<0\}$ shows that $g$ and $g^+$ have the same value function. The main case of interest is setting (a) with $\beta=0$, in which \eqref{eq:defV} reduces to \eqref{eq:defv}. The reader may specialise to it throughout; Section~\ref{sec:sball} is set in this case.

\begin{remark}\label{rem:recurrent}
For Brownian motion, property \ref{P:trans} below fails in setting (b) with $\beta=0$ and $d=2$. In that case Brownian motion is neighbourhood recurrent and $V\equiv\max g$, which does not vanish at infinity unless $g\equiv0$. Correspondingly, no ceiling in the sense of Definition~\ref{def:ceiling} below exists unless $g\equiv0$, since a superharmonic function on $\R^2$ which is bounded below is constant.
\end{remark}

For a function $u$ on $\overline\Lambda$,
\[
\omega_u(t):=\sup\{|u(x)-u(y)| : x,y\in\overline\Lambda,\ |x-y|\le t\},\qquad t\ge0,
\]
denotes its modulus of continuity, and $\norm{u}_\infty$ its supremum norm, also written $\norm{u}$. Under Assumption~\ref{ass:breg}, $g$ is uniformly continuous on $\overline\Lambda$, so that $\omega_g(t)\to0$ as $t\downarrow0$: in setting (a) because $\overline\Lambda$ is compact, and in setting (b) because $g$ is continuous and vanishes at infinity.

\subsection{Standing properties of the process}\label{sec:dpt}

All results are proved from the following properties of the family $(\PP^z)$.

\begin{assumption}\label{ass:process}
\begin{enumerate}[label=(P\arabic*)]
\item\label{P:markov} \emph{Markov.} $(\PP^z)$ is a strong Markov family, and $\mathcal F_0$ is $\PP^z$-trivial for every $z$ (Blumenthal's zero--one law).
\item\label{P:cont} \emph{Continuous martingales.} For every $z$, every $(\mathcal F_t)$-martingale under $\PP^z$ has a continuous version.
\item\label{P:tv} \emph{Smoothing.} For every $t>0$, the map $z\mapsto P_t(z,\cdot)$ is continuous in total variation, $\norm{\mu-\mu'}_{TV}:=\sup\{|\int f\,d\mu-\int f\,d\mu'|:f\text{ Borel},\ |f|\le1\}$.
\item\label{P:small} \emph{Small times and exit times.} For every bounded set $K\subset\R^d$ and every $r>0$,
\begin{align*}
\delta^K_r(\epsilon)&:=\sup_{w\in K}\PP^w\Bigl(\sup_{t\le\epsilon}|Y_t-w|\ge r\Bigr)\to0&&\text{as }\epsilon\downarrow0,\\
e^K_r&:=\beta\sup\bigl\{\E^z[\tau^+_{B^\circ(x;r)}]:\ x\in K,\ z\in B(x;r)\bigr\}\to0&&\text{as }r\downarrow0;
\end{align*}
and for every bounded open $U$, $\sup_z\E^z[\tau^+_U]<\infty$. In every use below the starting points range over a bounded set $K$ which is clear from the context, and we write $\delta_r$ and $e_r$ for $\delta^K_r$ and $e^K_r$. Only these local bounds are used, so drifts which are unbounded but locally bounded are allowed.
\item\label{P:polar} \emph{Polar points.} $\PP^z(\tau_{\{x\}}<\tau_{\pd\Lambda})=0$ for all $x\in\Lambda$ and $z\in\overline\Lambda\setminus\{x\}$.
\item\label{P:corr} \emph{Correctors.} For every $R>0$ there is a Borel function $\varphi_R\ge1$ on $\R^d$, bounded on bounded sets, which is $\beta$-harmonic and continuous in an open set containing $B(0;R)$, and satisfies $\varphi_R\le\bar\varphi_R$ on $B(0;R)$ for a constant $\bar\varphi_R$.
\item\label{P:trans} \emph{Decay at infinity.} In setting (b), for every $\rho>0$, $\E^x\bigl[e^{-\beta\tau_{B(0;\rho)}};\,\tau_{B(0;\rho)}<\infty\bigr]\to0$ as $|x|\to\infty$.
\end{enumerate}
\end{assumption}

Two of the properties are substantive restrictions on the dynamics: polarity \ref{P:polar} and decay at infinity \ref{P:trans}. Decay at infinity \ref{P:trans} concerns only setting (b). It is necessary and sufficient for ceilings (Definition~\ref{def:ceiling}) to exist for every gain satisfying Assumption~\ref{ass:breg} (Lemma~\ref{lem:ceiling}(iii) and the paragraph following its proof). When $\beta=0$ it implies transience, $|Y_t|\to\infty$ almost surely (Lemma~\ref{lem:transience}), which makes ceilings vanish along paths in Lemmas~\ref{lem:onestep}(i), \ref{lem:fatten} and~\ref{lem:verification}.

\paragraph{Brownian motion.} For Brownian motion all seven properties are classical, with \ref{P:trans} requiring $\beta>0$ or $d\ge3$; we verify them in Remark~\ref{rem:bmprops} in Appendix~\ref{app:conseq}. Moreover, $\beta$-harmonicity in the sense above coincides with the classical notion $\frac12\Delta u=\beta u$, and $\beta$-superharmonicity of a continuous function, defined in Section~\ref{sec:conseq}, with $\frac12\Delta u\le\beta u$ in the sense of distributions (Remark~\ref{rem:bmclassical} in Appendix~\ref{app:conseq}).

\subsection{Consequences}\label{sec:conseq}

The following facts are standard for Brownian motion, and a reader interested only in that case may take them as known. Their proofs, which use only the Markov property \ref{P:markov}, smoothing \ref{P:tv} and the small-time bounds \ref{P:small}, are given in Appendix~\ref{app:conseq}, together with elementary facts about $\beta$-superharmonic functions (Lemma~\ref{lem:superharm}).

\begin{lemma}[Martingales]\label{lem:mart}
If $u$ is $\beta$-harmonic in $O$ and $G$ is open with compact closure in $O$, then $e^{-\beta(t\wedge\tau^+_G)}u(Y_{t\wedge\tau^+_G})$ is a bounded $\PP^z$-martingale for $z\in G$. For $U$ bounded and open and $F$ Borel and bounded on bounded sets, $h^U_F$ is Borel on $\R^d$ and $\beta$-harmonic in $U$.
\end{lemma}

\begin{lemma}[Local to global]\label{lem:local}
Let $O$ be open and $u$ Borel and bounded on bounded sets. Suppose that every $z\in O$ is the centre of an open ball $B_z$ with compact closure in $O$ such that $u=h^{B_z}_u$ on $B_z$. Then $u$ is $\beta$-harmonic in $O$.
\end{lemma}

\begin{lemma}[Interior continuity]\label{lem:interior}
Let $u$ be $\beta$-harmonic in $B^\circ(z;r)$ with $|u|\le M$ there, and $y\in B(z;r/4)$. Then for every $\epsilon>0$
\[
|u(y)-u(z)|\le4M\,\delta_{r/4}(\epsilon)+M\,\|P_\epsilon(y,\cdot)-P_\epsilon(z,\cdot)\|_{TV}.
\]
Consequently $|u(y)-u(z)|\le M\,\Theta_{z,r}(|y-z|)$, where $\Theta_{z,r}(s)\to0$ as $s\downarrow0$, and $\Theta_{z,r}$ depends only on $z$, $r$ and $(\PP^w)$. In particular, $h^U_F$ is continuous in $U$.
\end{lemma}

\begin{lemma}[Regular points]\label{lem:reg}
Let $D\subset\R^d$ be open and $y\in\pd D$.
\begin{enumerate}[label=(\roman*)]
\item If $y$ is regular for $D$, then $\PP^z(\tau^+_D>\delta)\to0$ as $z\to y$, for every $\delta>0$. If $y$ is not regular, this fails for some $\delta>0$ along a sequence $z\to y$ in $D$.
\item If $D$ is bounded, $y$ is regular for $D$, and $F$ is Borel, bounded on bounded sets and continuous at $y$, then $h^D_F(z)\to F(y)$ as $z\to y$.
\end{enumerate}
\end{lemma}

\begin{lemma}[Admissible times]\label{lem:admissible}
Let $x\in\overline\Lambda$, and work under $\PP^x$.
\begin{enumerate}[label=(\roman*)]
\item In setting (a), the first hitting time by $X^x$ of a closed set $E\supset\pd\Lambda$ is admissible.
\item In setting (a), let $O\subset\Lambda$ be open and $\tau^+_O:=\inf\{t>0:Y_t\notin O\}$. Then $\tau^+_O\le\tau_{\pd\Lambda}$ almost surely. More generally, if $\sigma$ is admissible, then $\vartheta':=\inf\{t>\sigma:Y_t\notin O\}\le\tau_{\pd\Lambda}$ almost surely on $\{\sigma<\infty\}$.
\item In both settings, let $O$ be a bounded open set, $\sigma$ an admissible stopping time and $\vartheta':=\inf\{t>\sigma:Y_t\notin O\}$. Then $\tau^+_O<\infty$ almost surely, and $\vartheta'<\infty$ almost surely on $\{\sigma<\infty\}$.
\end{enumerate}
In setting (b) every stopping time is admissible.
\end{lemma}

We will say that a bounded function $C\in C(\overline\Lambda)$ is \emph{$\beta$-superharmonic} in $\Lambda$ if, for every $x\in\overline\Lambda$, the process $e^{-\beta(t\wedge\tau_{\pd\Lambda})}C(X^x_t)$ is a supermartingale; recall that $X^x$ is $Y$ under $\PP^x$, stopped at $\tau_{\pd\Lambda}$. The definition uses martingales, rather than averaging over exit distributions as for $\beta$-harmonicity in Section~\ref{sec:process}, because superharmonicity is needed only for ceilings, and these are studied via optional sampling up to $\tau_{\pd\Lambda}$.

\subsection{Admissible replacements}\label{sec:dipo}

For an open set $A\subset\Lambda$, let $\Gamma_A$ denote the family of bounded open sets $\gamma\subset A$.

\begin{definition}\label{def:ceiling}
A \emph{ceiling} is a continuous function $C:\overline\Lambda\to\R$ which is $\beta$-superharmonic in $\Lambda$ and satisfies $C\ge g$ on $\overline\Lambda$, together with $C=g$ on $\pd\Lambda$ in setting (a), and $C(x)\to0$ as $|x|\to\infty$ in setting (b). In the language of optimal stopping, a ceiling is a continuous $\beta$-excessive majorant of $g$ with prescribed boundary behaviour; in particular every ceiling dominates $V$, by optional sampling (Lemma~\ref{lem:superharm}(ii)). Its \emph{non-contact set} is the set $A_C:=\{x\in\Lambda : C(x)>g(x)\}$, which is open since $C$ and $g$ are continuous. The elements of $\Gamma_{A_C}$ are the \emph{admissible sets} for $C$. For $\gamma\in\Gamma_{A_C}$, the \emph{harmonic replacement} of $C$ on $\gamma$ is $C_\gamma:=h^\gamma_C$ on $\overline\Lambda$, which is well defined since the path leaves $\gamma$ at a point of $\overline\gamma\subset\overline\Lambda$ (in particular $C_\emptyset=C$, since $\tau^+_\emptyset=0$). The \emph{admissible replacements} for $C$ are the elements of
\begin{equation}\label{eq:chac}
\cH(C) := \bigl\{\, C_\gamma :\ \gamma\in\Gamma_{A_C},\ C_\gamma\ge g \,\bigr\},
\end{equation}
and the \emph{envelope} of $C$ is their pointwise infimum,
\begin{equation}\label{eq:envdef}
\Env[C] := \inf\cH(C)\qquad\text{on }\overline\Lambda.
\end{equation}
\end{definition}

Ceilings are bounded, and in setting (b) they are nonnegative since $g\ge0$. A harmonic replacement $C_\gamma$ is continuous in $\gamma$ by Lemma~\ref{lem:interior}, and equals $C$ off $\overline\gamma$ and at the points of $\pd\gamma$ that are regular for $\gamma$. At an irregular point of $\pd\gamma$ the path enters $\gamma$ at once, and $C_\gamma$ is the value of continuing into $\gamma$, so $C_\gamma$ need not be continuous on $\overline\Lambda$. Continuity of $C_\gamma$ on $\overline\Lambda$ is not needed, since the envelopes are shown to be continuous directly (Proposition~\ref{prop:contenv}). Admitting all bounded open sets, rather than only regular ones, is what makes the comparison lemma (Lemma~\ref{lem:comparison}) usable: the admissible sets it produces have the form $\{u<C\}$, whose boundaries are not controlled. By Lemma~\ref{lem:onestep} below, each admissible replacement majorises $g$, lies below $C$ and has the discounted supermartingale property. Thus $\Env[C]$ is a Perron-type infimum over a structured family (Figure~\ref{fig:harm}).

\begin{figure}
    \centering
    \input{harmonic_replacement.pgf}
    \caption{An admissible replacement $C_\gamma\in\cH(C)$, for Brownian motion on the unit disc with $\beta=0$, the ceiling $C=w_0$ of \eqref{eq:w0} with $\rho=0.6$, and a gain $g$ supported in $B(0;0.58)$. (a) An open subset $\gamma$ of the non-contact set $A_C$ (shaded), and contours of $g$. (b) $C$ (dashed), $C_\gamma$ and $g$ along the dotted diameter in (a). On $\gamma$ the values of $C$ are replaced by the harmonic extension of its boundary values; the replacement still majorises $g$, and nearly touches it at the marked point.}
    \label{fig:harm}
\end{figure}

\begin{remark}\label{rem:interp}
If $C_\gamma\in\cH(C)$ and $x\in\gamma$, then $C_\gamma(x)$ is the expected discounted reward from continuing until $X^x$ leaves $\gamma$ and then collecting $C$; for $x\notin\overline\gamma$, and in particular for $\gamma=\emptyset$, it is the reward from collecting $C$ immediately. When $C$ is (approximately) the value of a stopping rule, $C_\gamma$ is the value of the composed rule which continues until the exit from $\gamma$ and then follows that rule, and the constraint $C_\gamma\ge g$ says that immediate stopping is never better than this composition. The iteration of Section~\ref{sec:canen} therefore encodes compositions of stopping times.
\end{remark}

\paragraph{Examples of ceilings.} Ceilings exist for general gains and processes in both settings (Lemma~\ref{lem:ceiling} in Appendix~\ref{app:ceilings}). The simplest are available when $g\ge0$ has compact support; write $\bar g:=\max g$. In setting (a), let $e(x):=\E^x[\tau_{\pd\Lambda}]$ be the expected exit time, which is continuous on $\overline\Lambda$, positive in $\Lambda$ and zero on $\pd\Lambda$ (Lemma~\ref{lem:exittime}). If $\supp(g)$ is contained in a nonempty compact set $K\subset\Lambda$, then
\begin{equation}\label{eq:ceile}
\bar g\,\frac{e}{\min_Ke}
\end{equation}
is a ceiling: it is continuous, it is $\beta$-superharmonic because $e(X^x_t)$ is a nonnegative supermartingale, it is at least $\bar g$ on $K$, and it vanishes on $\pd\Lambda$. For Brownian motion on the unit ball, $e(x)=(1-|x|^2)/d$, and with $K=B(0;\rho)$, $\rho<1$, the ceiling \eqref{eq:ceile} is
\begin{equation}\label{eq:w0}
w_0(x) := \bar g\,\frac{1-|x|^2}{1-\rho^2}.
\end{equation}
For Brownian motion in setting (b), with $\beta>0$ or $d\ge3$, let
\begin{equation}\label{eq:qrho}
q_\rho(x):=\begin{cases}\min\{1,\,e^{-\sqrt{2\beta}(|x|-\rho)}\}, & \beta>0,\\[0.5ex] \min\{1,\,(\rho/|x|)^{d-2}\}, & \beta=0,\ d\ge3.\end{cases}
\end{equation}
Here $q_\rho(0):=1$. If $\supp(g)\subset B(0;\rho)$, then $\bar g\,q_\rho$ is a ceiling.

\subsection{Iterated envelopes $w_n$}\label{sec:canen}

Fix an initial ceiling $w_0$, for instance one of those above or in Appendix~\ref{app:ceilings}. Define recursively, as long as $w_{n-1}$ is a ceiling,
\begin{equation}\label{eq:defchnn}
w_n := \Env[w_{n-1}], \qquad A_n := A_{w_n} = \{x\in\Lambda : w_n(x)>g(x)\},\qquad n\ge1,
\end{equation}
and $A_0:=A_{w_0}$. Since $C_\emptyset=C$ and admissible replacements majorise $g$, we have $g\le w_n\le w_{n-1}$ whenever $w_n$ is defined, so $w_n$ decreases to a limit $w_\infty$, and we set $A_\infty := \{x\in\Lambda : w_\infty(x)>g(x)\}$. The main result of the next section is that each $w_n$ is again a ceiling, so that the recursion never stops. Thus $w_1$ is constructed on $A_0$ with ceiling $w_0$, and $w_n$ is constructed only on the previous non-contact set $A_{n-1}$, with ceiling $w_{n-1}$.

For $n\in\Nb:=\{0,1,2,\dots\}\cup\{\infty\}$ and $x\in A_n$, write $A_n^x$ for the connected component of $A_n$ containing $x$, and set $\pd A_n^x:=\{x\}$ for $x\in\overline\Lambda\setminus A_n$. Write
\begin{align}\label{eq:taust}
\tau_n := \inf\{t \geq 0: X_t^x \in \pd A_n^x\},
\end{align}
which for $x\in A_n$ is the first exit time of $X^x$ from $A_n$, and is zero otherwise. We call $\tau_n$ the \emph{$n$th rule}, and define the \emph{value of the $n$th rule} by
\begin{align}\label{eq:defwb}
    \overline{w}_n(x) = \E\bigl[e^{-\beta\tau_n}g\bigl(X^x_{\tau_n}\bigr)\1_{\{\tau_n<\infty\}}\bigr].
\end{align}
Its value coincides with $g$ on $\overline\Lambda\setminus A_n$. These definitions, and \eqref{eq:epsrule} below, use that $w_n$ is continuous, so that $A_n$ is open; this is part of Theorem~\ref{thm:main}(i). Given $\varepsilon>0$, we also define the \emph{$n$th $\varepsilon$-rule} and its value by
\begin{equation}\label{eq:epsrule}
\begin{aligned}
\sigma_n &:= \inf\{t\ge0 : w_n(X^x_t)\le g(X^x_t)+\varepsilon\},\\
v_n(x)&:=\E\bigl[e^{-\beta\sigma_n}g(X^x_{\sigma_n})\1_{\{\sigma_n<\infty\}}\bigr].
\end{aligned}
\end{equation}

\subsection{Polarity and localisation}
\label{sec:sball}

The affine majorants in the Dynkin--Yushkevich representation are global. Our goal in this subsection is therefore to motivate the localisation in the construction.

Polarity plays two roles. It creates the problem that localisation solves: since points are polar, the value function hardly feels a gain concentrated near a point, whereas a harmonic majorant on all of $\Lambda$ must, by Harnack's inequality. It is also part of the solution: a replacement may leave out a small ball around such a point (Lemma~\ref{lem:spike}), and the fattening lemma (Lemma~\ref{lem:fatten}) may add a small ball to a replacement set, because few paths reach a small ball.

For simplicity
we assume setting (a) with $\beta=0$, Brownian motion, $g\ge0$ and $\Lambda$ connected. The natural global majorant by harmonic functions is
\[
W_{\rm glob}(x) := \inf\bigl\{h(x) : h\text{ harmonic in } \Lambda,\ h \ge g \text{ on } \Lambda\bigr\}.
\]
Every such $h$ is nonnegative, since $h\ge g\ge0$. For a connected open set $D$ and $x,y\in D$ let
\[
c_D(x,y):=\inf\bigl\{h(x)/h(y):\ h>0\text{ harmonic in }D\bigr\},
\]
which is positive by Harnack's inequality, and write $c:=c_\Lambda$.

Proposition~\ref{prop:global} shows that Harnack rigidity transfers to the global majorant. For a thinning spike, Lemma~\ref{lem:spike} shows that, off the origin, the value function and the first envelope tend to $0$, whereas the global majorant stays bounded away from $0$.

\begin{proposition}[Harnack rigidity of $V$]\label{prop:global}
$W_{\rm glob}(x)\ge c(x,y)\,V(y)$ for all $x,y\in\Lambda$.
\end{proposition}

\begin{proof}
Let $h$ be harmonic in $\Lambda$ with $h\ge g$. We first show that $h\ge V$. Let $\Lambda_k$ be open with compact closures in $\Lambda$ and $\Lambda_k\uparrow\Lambda$, with exit times $\tau_k\uparrow\tau_{\pd\Lambda}$, and let $\sigma\le\tau_{\pd\Lambda}$ be a stopping time. By optional sampling, $h(x)=\E^x[h(Y_{\sigma\wedge\tau_k})]$. Since $h\ge g\ge0$ on $\Lambda$ and $g$ is continuous on $\overline\Lambda$, Fatou's lemma gives $h(x)\ge\E^x[\liminf_kg(Y_{\sigma\wedge\tau_k})]=\E[g(X^x_\sigma)]$. Hence $h\ge V$. If $h\equiv0$ then $V\equiv0$ and there is nothing to prove; otherwise $h>0$ by the minimum principle, and $h(x)\ge c(x,y)h(y)\ge c(x,y)V(y)$. Taking the infimum over $h$ gives the claim.
\end{proof}

\begin{lemma}\label{lem:spike}
Let $0\in\Lambda$ and, for $0<\varepsilon<\operatorname{dist}(0,\pd\Lambda)$, let $g^\varepsilon(x):=(1-|x|/\varepsilon)_+$, with value function $V^\varepsilon$ and global majorant $W^\varepsilon_{\rm glob}$, and let $C^\varepsilon$ be the ceiling \eqref{eq:ceile} for $g^\varepsilon$, with $K:=B(0;\varepsilon)$. Then for every $x\in\Lambda\setminus\{0\}$,
\[
W^\varepsilon_{\rm glob}(x)\ge c(x,0)\quad\text{for every }\varepsilon,
\]
whereas
\[
V^\varepsilon(x)\le\Env[C^\varepsilon](x)\to0\qquad(\varepsilon\to0).
\]
In particular, the global majorant and the value function need not coincide, while the gap between the first envelope and the value function vanishes as the spike thins.
\end{lemma}

\begin{proof}
Since $0\le g^\varepsilon\le1=g^\varepsilon(0)$, we have $V^\varepsilon(0)=1$, and Proposition~\ref{prop:global} with $y=0$ gives $W^\varepsilon_{\rm glob}(x)\ge c(x,0)$. Next, $V^\varepsilon\le\Env[C^\varepsilon]$ by Lemma~\ref{lem:onestep}(ii). Fix $x\in\Lambda\setminus\{0\}$, let $\varepsilon_0:=\operatorname{dist}(0,\pd\Lambda)/2$, take $\varepsilon<\min(\varepsilon_0,|x|)$, and write $C:=C^\varepsilon$ and $\gamma:=\Lambda\setminus B(0;\varepsilon)$, so that $x\in\gamma$. Since $e>0$ in $\Lambda$, we have $C>0=g^\varepsilon$ on $\gamma$, so $\gamma\in\Gamma_{A_C}$. Moreover $C_\gamma\ge0=g^\varepsilon$ on $\gamma$, and $C_\gamma=C\ge g^\varepsilon$ off $\gamma$, so $C_\gamma\in\cH(C)$. A path from $x$ leaves $\gamma$ either at $\pd\Lambda$, where $C=0$, or on reaching $B(0;\varepsilon)$. Hence
\[
\Env[C](x)\le C_\gamma(x)\le\norm{C}\,\PP^x\bigl(\tau_{B(0;\varepsilon)}<\tau_{\pd\Lambda}\bigr),\qquad\norm{C}\le\frac{\max e}{\min_{B(0;\varepsilon_0)}e}.
\]
The probability tends to $0$ as $\varepsilon\to0$, since points are polar, by continuity of paths (as in the proof of Lemma~\ref{lem:fatten}).
\end{proof}

Lemma~\ref{lem:spike} shows what localisation gains. A replacement on $\gamma$ need majorise $g$ only on $\gamma$, and off $\gamma$ it equals the previous iterate. A domain that excludes the spike therefore frees the replacement from the Harnack constraint of Proposition~\ref{prop:global}, which ties values at $x$ to the tip of the spike. The gain then enters only through the exit position, since the replacement averages the previous iterate over the exit from $\gamma$, and by polarity \ref{P:polar} few paths reach the spike. Localised majorants nevertheless still majorise $V$ (Theorem~\ref{thm:main}(i)).

\section{Main results}
\label{sec:branching}

\begin{theorem}\label{thm:main}
Let Assumptions~\ref{ass:breg} and~\ref{ass:process} hold, let $w_0$ be a ceiling, and let $w_n:=\Env[w_{n-1}]$ be the iterated envelopes~\eqref{eq:defchnn}. Then:
\begin{enumerate}[label=(\roman*)]
\item each $w_n$ is a ceiling, each rule $\tau_n$ is admissible with value $\ow_n$, and $\ow_n\le\ow_{n+1}\le V\le w_{n+1}\le w_n$ for $n\ge0$;
\item $w_\infty=V$, and the convergence $w_n\downarrow V$ is uniform on $\overline\Lambda$; in particular the limit contact set $\{w_\infty=g\}$ is the stopping region $\{V=g\}$, and the first entry time $\tau_\infty$ into it is optimal, that is, $\ow_\infty=V$;
\item for every $\varepsilon>0$, each $\varepsilon$-rule $\sigma_n$ is admissible with value $v_n\le V$, and $\norm{w_n-v_n}_\infty<2\varepsilon$ for all sufficiently large $n$.
\end{enumerate}
\end{theorem}

Part (iii) gives certified near-optimal rules after finitely many steps. Since $v_n\le V\le w_n$, the rule $\sigma_n$ is $\norm{w_n-v_n}_\infty$-optimal from every starting point, and this gap is computed from the construction alone. Therefore, iterating until $\norm{w_n-v_n}_\infty<2\varepsilon$ produces, after finitely many steps and without knowledge of $V$, a stopping rule which is $2\varepsilon$-optimal from every starting point. The exact rules $\tau_n$ give monotone lower bounds by (i), while the $\varepsilon$-rules $\sigma_n$ carry certificates.

Each step of the construction takes an infimum of solutions of linear Dirichlet problems, on subdomains of the current non-contact set which are chosen by the construction, and every iterate majorises the value function. The theorem requires no regularity beyond Assumptions~\ref{ass:breg} and~\ref{ass:process}: neither regularity of the free boundary nor continuity of the value $\ow_\infty$ of the limit rule is assumed, and the latter is a consequence of (ii). The limit does not depend on the initial ceiling $w_0$. Continuity of $g$ cannot be dropped: for Brownian motion in the unit ball with $\beta=0$, and $g$ the indicator of the centre, which is polar, $V$ vanishes off the centre while $V(0)=1$, so $V$ is not a uniform limit of continuous functions.

For Brownian motion, translation invariance gives additional regularity: every envelope, and the value function, inherits the modulus of continuity of the data, uniformly in $n$.

\begin{theorem}\label{thm:modulus}
For Brownian motion, let Assumption~\ref{ass:breg} hold and let $w_0$ be a ceiling. Then, with $\omega:=\max(\omega_{w_0},\omega_g)$,
\[
|w_n(x)-w_n(y)|\le\omega(|x-y|)\quad\text{and}\quad|V(x)-V(y)|\le\omega(|x-y|)
\]
for all $x,y\in\overline\Lambda$ and $n\ge0$. In particular, if $g$ is $L$-Lipschitz and $w_0$ is $L_0$-Lipschitz, then each $w_n$, and $V$, are $\max(L,L_0)$-Lipschitz.
\end{theorem}

For the explicit ceilings one may take $L_0=2\bar g/(1-\rho^2)$ for \eqref{eq:w0}, and, for $\bar g\,q_\rho$, $L_0=\sqrt{2\beta}\,\bar g$ if $\beta>0$ and $L_0=(d-2)\bar g/\rho$ if $\beta=0$, by bounding their gradients.

Theorem~\ref{thm:main} is proved in Section~\ref{sec:proofs}, following the one-step lemmas of Section~\ref{sec:esp}, and Theorem~\ref{thm:modulus} in Section~\ref{sec:translation}.

\subsection{One-step lemmas}\label{sec:esp}

Throughout this subsection, $C$ is a ceiling, $A:=A_C$ and $W:=\Env[C]$.

The first lemma records the basic properties of a single step, which are those one would expect. A harmonic replacement of a ceiling lies between $g$ and $C$ and inherits the supermartingale property of $C$, as in the classical fact that Poisson modification preserves superharmonicity; the envelope $W$ therefore lies between $V$ and $C$. The proof is a routine application of optional sampling and the strong Markov property, and is given in Appendix~\ref{app:conseq}; in setting (b) with $\beta=0$ it uses transience (Lemma~\ref{lem:transience}).

\begin{lemma}\label{lem:onestep}
\begin{enumerate}[label=(\roman*)]
\item Each $C_\gamma\in\cH(C)$ satisfies $g\le C_\gamma\le C$ on $\overline\Lambda$, and
\[
\E\bigl[e^{-\beta\sigma}C_\gamma(X^x_\sigma)\1_{\{\sigma<\infty\}}\bigr]\le C_\gamma(x)
\]
for every $x\in\overline\Lambda$ and every admissible stopping time $\sigma$.
\item $g\le W\le C$ on $\overline\Lambda$, and $W=g=C$ on $\overline\Lambda\setminus A$. Moreover $W\ge V$, and if $W$ is continuous then $W$ is a ceiling.
\end{enumerate}
\end{lemma}

The next result concerns functions $u=h^U_F$, $\beta$-harmonic in a bounded open set $U$ with boundary data $F$. It strengthens the classical pasting lemma, by which the minimum of a superharmonic function and a harmonic function dominating it at the boundary is superharmonic \cite[Corollary~3.2.4]{ArmitageGardiner}, in two ways. Its hypotheses are more general: only the data $F$ need dominate $C$ off $U$, rather than $u$ itself in the limit at every boundary point (including irregular points). The superharmonicity of $C$ is not used; rather, the other properties of a ceiling and martingale continuity \ref{P:cont} play the key roles. Its conclusion is stronger: $\min(u,C)$ is identified with the harmonic replacement of $C$ on $\{u<C\}$. In particular this is true at irregular points of the boundary of $\{u<C\}$, where the replacement takes continuation values. This makes $\min(u,C)$ an admissible replacement when $u\ge g$, hence an upper bound for the envelope.

\begin{lemma}[Comparison]\label{lem:comparison}
Let $U\subset\R^d$ be bounded and open, let $F:\R^d\to\R$ be Borel and bounded on bounded sets, and let $u:=h^U_F$. Suppose that
\begin{enumerate}[label=(H\arabic*)]
\item $u\ge g$ on $U\cap\overline\Lambda$, and
\item $F\ge C$ on $\overline\Lambda\setminus U$.
\end{enumerate}
Let $O:=\{x\in U\cap A:u(x)<C(x)\}$. Then $O\in\Gamma_A$, $C_O\in\cH(C)$, and $C_O=\min(u,C)$ on $\Lambda$. In particular $W\le u$ on $\Lambda$.
\end{lemma}

\begin{proof}
By Lemma~\ref{lem:interior}, $u$ is continuous in $U$, so $O$ is open; it is bounded and contained in $A$, so $O\in\Gamma_A$. Let $\Phi:=u$ on $U$ and $\Phi:=F$ off $U$. Fix $x\in\overline\Lambda$, work under $\PP^x$, and write $Z:=Y$ and $\tau:=\tau^+_U$, which is finite almost surely by Lemma~\ref{lem:admissible}(iii), and let $N_t:=\E[e^{-\beta\tau}F(Z_\tau)\,|\,\mathcal F_t]$. Then $N$ is bounded, since $Z_\tau\in\overline U\cup\{x\}$; it is continuous by \ref{P:cont}; and $N_0=u(x)$ by \ref{P:markov}. By optional sampling and the strong Markov property, $N_\vartheta=e^{-\beta\vartheta}\Phi(Z_\vartheta)$ almost surely for every stopping time $\vartheta$ with $0<\vartheta\le\tau$. Indeed, on $\{\vartheta<\tau\}$ we have $Z_\vartheta\in U$, since $0<\vartheta<\tau$, and the remaining exit time is that of the restarted path; on $\{\vartheta=\tau\}$ we have $Z_\vartheta\notin U$, since $U$ is open and paths are continuous. The restriction $\vartheta>0$ matters: at a point $x\notin U$ from which the path enters $U$ immediately, $N_0=u(x)$ may differ from $F(x)$. In particular, almost surely $N_t=e^{-\beta t}u(Z_t)$ for all $t\in(0,\tau)$ simultaneously, since both sides are continuous there and agree almost surely at each rational $t$.

\emph{Step 1.} Suppose that $\tau^+_O>0$ almost surely. We show that $C_O(x)=u(x)$ and $g(x)\le u(x)\le C(x)$. Let $\vartheta:=\tau^+_O$, so that $0<\vartheta\le\tau$ since $O\subset U$, and $\zeta:=Z_\vartheta\in\overline O\subset\overline\Lambda$. For $0<t<\vartheta$ we have $Z_t\in O$, so $N_t=e^{-\beta t}u(Z_t)<e^{-\beta t}C(Z_t)$. Letting $t\uparrow\vartheta$, continuity of $N$ and of $C$ on $\overline\Lambda$ gives $e^{-\beta\vartheta}\Phi(\zeta)=N_\vartheta\le e^{-\beta\vartheta}C(\zeta)$. For the reverse inequality note that $\zeta\notin O$, and there are three cases. If $\zeta\in U\cap A$, then $\Phi(\zeta)=u(\zeta)\ge C(\zeta)$ by definition of $O$. If $\zeta\in U\setminus A$, then $\zeta\in\overline\Lambda\setminus A$, where $C=g$ by Lemma~\ref{lem:onestep}(ii), so $C(\zeta)=g(\zeta)\le u(\zeta)=\Phi(\zeta)$ by (H1). If $\zeta\notin U$, then $\Phi(\zeta)=F(\zeta)\ge C(\zeta)$ by (H2). Hence $\Phi(\zeta)=C(\zeta)$ almost surely, and
\[
C_O(x)=\E\bigl[e^{-\beta\vartheta}C(\zeta)\bigr]=\E[N_\vartheta]=N_0=u(x).
\]
Moreover $Z_t\in O$ for small $t>0$, so $u(x)=N_0=\lim_{t\downarrow0}e^{-\beta t}u(Z_t)$ lies between $\lim_{t\downarrow0}g(Z_t)=g(x)$ and $\lim_{t\downarrow0}C(Z_t)=C(x)$, by (H1) and the definition of $O$. Since these are inequalities between numbers, it suffices to derive them along one path, and almost every path will do.

\emph{Step 2: $C_O\in\cH(C)$.} Let $x\in\overline\Lambda$. By \ref{P:markov}, either $\tau^+_O=0$ almost surely, and then $C_O(x)=C(x)\ge g(x)$; or $\tau^+_O>0$ almost surely, and then $C_O(x)=u(x)\ge g(x)$ by Step~1.

\emph{Step 3: $C_O=\min(u,C)$ on $\Lambda$.} Let $x\in\Lambda$. If $\tau^+_O>0$ almost surely, then $C_O(x)=u(x)\le C(x)$ by Step~1. Otherwise $C_O(x)=C(x)$, and we show that $C(x)\le u(x)$. If $\tau=0$ almost surely, then $x\notin U$ and $u(x)=F(x)\ge C(x)$ by (H2). Otherwise $\tau>0$ almost surely, and since $\Lambda$ is open, almost surely there are times $t_k\downarrow0$ with $Z_{t_k}\in(U\cap\Lambda)\setminus O$. At such points $u\ge C$: by definition of $O$ if they lie in $A$, and by (H1), since $C=g$ there, otherwise. Hence
\[
u(x)=N_0=\lim_kN_{t_k}=\lim_ke^{-\beta t_k}u(Z_{t_k})\ge\lim_kC(Z_{t_k})=C(x).
\]
Finally $W\le C_O=\min(u,C)\le u$ on $\Lambda$.
\end{proof}

When $U\subset\Lambda$, only the values of $F$ on $\overline\Lambda$ matter, and we then specify $F$ only there.

The next lemma establishes controlled monotonicity in the replacement set: enlarging the set lowers the replacement, by at most $2\norm{C}$ times the probability of hitting the added part. It also shows that majorisation of $g$ need only be checked on the replacement set itself; at irregular boundary points, where the replacement takes continuation values, it follows from continuity of $g$.

\begin{lemma}[Monotonicity]\label{lem:monotone}
Let $O\subset\tilde O$ be bounded open subsets of $\Lambda$, write $C_U:=h^U_C$ for every bounded open $U\subset\Lambda$, and let $\tau:=\tau^+_O$.
\begin{enumerate}[label=(\roman*)]
\item $C_{\tilde O}\le C_O\le C$ on $\overline\Lambda$, and for every $z\in\overline\Lambda$
\[
C_O(z)-C_{\tilde O}(z)\le2\norm{C}\,\E^z\bigl[e^{-\beta T};\,T<\tau_{\pd\Lambda}\bigr],\qquad T:=\tau_{\tilde O\setminus O}.
\]
If $z\in\tilde O$, or if $\tau>0$ $\PP^z$-almost surely, then moreover
\begin{equation}\label{eq:monoid}
C_O(z)-C_{\tilde O}(z)=\E^z\bigl[e^{-\beta\tau}(C-C_{\tilde O})(Y_\tau);\,Y_\tau\in\tilde O\bigr].
\end{equation}
\item If $\tilde O\subset A$ and $C_{\tilde O}\ge g$ on $\tilde O$, then $C_{\tilde O}\in\cH(C)$.
\end{enumerate}
\end{lemma}

\begin{proof}
(i) The proof of Lemma~\ref{lem:onestep}(i) in Appendix~\ref{app:conseq} shows that $C_U\le C$ for every bounded open $U\subset\Lambda$; it uses only the supermartingale property of $C$. Fix $z\in\overline\Lambda$ and let $\tilde\tau:=\tau^+_{\tilde O}$ under $\PP^z$. Suppose first that $z\in\tilde O$ or $\tau>0$ almost surely. The path lies in $O\subset\tilde O$ on $(0,\tau)$ (vacuously if $\tau=0$, in which case $z\in\tilde O$). If $Y_\tau\notin\tilde O$, then $\tau>0$, since otherwise $Y_\tau=z\in\tilde O$; so the path has left $\tilde O$ at time $\tau$, and $\tilde\tau=\tau$. If $Y_\tau\in\tilde O$, the path has not left the open set $\tilde O$ by time $\tau$, so by the strong Markov property at $\tau$ the remaining exit time is $\tau^+_{\tilde O}$ for the restarted path. Hence
\[
C_{\tilde O}(z)=\E^z\bigl[e^{-\beta\tau}\bigl(C(Y_\tau)\1_{\{Y_\tau\notin\tilde O\}}+C_{\tilde O}(Y_\tau)\1_{\{Y_\tau\in\tilde O\}}\bigr)\bigr],
\]
which gives \eqref{eq:monoid}. It is nonnegative since $C_{\tilde O}\le C$. On $\{Y_\tau\in\tilde O\}$ we have $Y_\tau\in\tilde O\setminus O$, so $T\le\tau$, and $\tau<\tau_{\pd\Lambda}$ since $Y_\tau\in\Lambda$. As $|C-C_{\tilde O}|\le2\norm{C}$, the bound follows. In the remaining case $z\notin\tilde O$ and $\tau=0$ almost surely, so $C_O(z)=C(z)\ge C_{\tilde O}(z)$. If $\tilde\tau=0$ the difference vanishes. Otherwise $\tilde\tau>0$ almost surely by \ref{P:markov}, so $z\in\Lambda$, since points of $\pd\Lambda$ are regular. The path lies in $\tilde O$ at all small positive times and outside $O$ at arbitrarily small positive times, so $T=0<\tau_{\pd\Lambda}$, and the bound holds since $|C-C_{\tilde O}|\le2\norm{C}$.

(ii) It suffices to show $C_{\tilde O}\ge g$ on $\overline\Lambda\setminus\tilde O$. Let $z$ be such a point. If $\tau^+_{\tilde O}=0$ almost surely, then $C_{\tilde O}(z)=C(z)\ge g(z)$. Otherwise $\tilde\tau:=\tau^+_{\tilde O}>0$ almost surely by \ref{P:markov}. By the Markov property at time $t>0$, $C_{\tilde O}(z)=\E^z[e^{-\beta(t\wedge\tilde\tau)}\Theta_t]$, where $\Theta_t:=C_{\tilde O}(Y_t)$ on $\{t<\tilde\tau\}$, with $Y_t\in\tilde O$, and $\Theta_t:=C(Y_{\tilde\tau})$ on $\{\tilde\tau\le t\}$. In both cases $\Theta_t\ge g(Y_{t\wedge\tilde\tau})$, so $C_{\tilde O}(z)\ge\E^z[e^{-\beta(t\wedge\tilde\tau)}g(Y_{t\wedge\tilde\tau})]\to g(z)$ as $t\downarrow0$, by continuity of $g$ and bounded convergence.
\end{proof}

An application of Lemma~\ref{lem:monotone} in setting (b) is the following truncation lemma, which confines admissible sets to a large ball at small cost.

\begin{lemma}[Truncation]\label{lem:truncation}
In setting (b), let $C_\gamma\in\cH(C)$ and $R>0$, and put $\gamma_R:=\gamma\cap B^\circ(0;R)$. Then $C_{\gamma_R}\in\cH(C)$ and
\[
C_{\gamma_R}\le C_\gamma+\sup_{|w|\ge R}C(w)\qquad\text{on }\R^d.
\]
\end{lemma}

\begin{proof}
By Lemma~\ref{lem:monotone}(i), $C_{\gamma_R}\ge C_\gamma\ge g$, so $C_{\gamma_R}\in\cH(C)$, as $\gamma_R\in\Gamma_A$. Let $\tau:=\tau^+_{\gamma_R}$. If $z\in\gamma$ or $\tau>0$ almost surely, \eqref{eq:monoid} gives $C_{\gamma_R}(z)-C_\gamma(z)=\E^z[e^{-\beta\tau}(C-C_\gamma)(Y_\tau);\,Y_\tau\in\gamma]$. On $\{Y_\tau\in\gamma\}$ we have $|Y_\tau|\ge R$; and $C_\gamma\ge g\ge0$. Hence the integrand is at most $\sup_{|w|\ge R}C(w)$. Otherwise $z\notin\gamma$ and $\tau=0$ almost surely, so $C_{\gamma_R}(z)=C(z)$ and, as $C_\gamma\ge0$, $C_{\gamma_R}(z)-C_\gamma(z)\le C(z)$. If $C_\gamma(z)=C(z)$ the difference vanishes. Otherwise $C_\gamma(z)<C(z)$, so $\tau^+_\gamma>0$ almost surely, and the path enters $\gamma$ immediately but not $\gamma_R$; as $\gamma$ and $\gamma_R$ coincide near every point of $B^\circ(0;R)$, this forces $|z|\ge R$, and $C(z)\le\sup_{|w|\ge R}C(w)$.
\end{proof}

The next two lemmas are used to establish the continuity of envelopes, by giving near-optimal replacements a common modulus of continuity at a point $x$. First, Lemma~\ref{lem:slack} uses Lemma~\ref{lem:comparison} to shift upward the harmonic part of an admissible replacement, giving it a margin above $g$ on its domain at an arbitrarily small cost. Then Lemma~\ref{lem:fatten} uses Lemma~\ref{lem:monotone} to add to each such domain a ball of fixed radius around $x$, on which the interior estimates of Lemma~\ref{lem:interior} apply. The margin keeps the enlarged replacement admissible: near $x$ the ceiling is strictly above $g$, and away from $x$ the ball is rarely hit, by polarity.

\begin{lemma}[Slack]\label{lem:slack}
Let $R>0$, let $C_\gamma\in\cH(C)$ with $\gamma\subset B^\circ(0;R)$, and let $\delta>0$. Then there is $O\in\Gamma_A$ with $O\subset\gamma$, $C_O\in\cH(C)$,
\[
C_O\ge g+\delta\ \text{ on }O,\qquad\text{and}\qquad C_O\le C_\gamma+\delta\varphi_R\ \text{ on }\Lambda.
\]
\end{lemma}

\begin{proof}
If $\gamma=\emptyset$ take $O=\emptyset$. Otherwise apply Lemma~\ref{lem:comparison} with $U:=\gamma$, and $F:=C+\delta\varphi_R$ on $\overline\Lambda$. We first show that $h^\gamma_{\varphi_R}=\varphi_R$ on $\overline\Lambda$. At points where $\tau^+_\gamma=0$ almost surely this is clear. Otherwise, by \ref{P:markov}, $\tau^+_\gamma>0$ almost surely, and the starting point $z$ lies in $\overline\gamma\subset B(0;R)$. Let $G$ be open with $\overline\gamma\subset G$ and $\overline G$ a compact subset of the open set in which $\varphi_R$ is $\beta$-harmonic. We have $\tau^+_\gamma\le\tau^+_G$, and optional sampling for the bounded martingale of Lemma~\ref{lem:mart}, at the almost surely finite time $\tau^+_\gamma$, gives the claim. Hence $u:=h^\gamma_F=C_\gamma+\delta\varphi_R$ on $\overline\Lambda$. Hypothesis (H1) of Lemma~\ref{lem:comparison} holds since $C_\gamma\ge g$ and $\varphi_R\ge0$, and (H2) holds since $\varphi_R\ge0$. So $O:=\{z\in\gamma\cap A:u(z)<C(z)\}$ satisfies $O\in\Gamma_A$, $C_O\in\cH(C)$ and $C_O=\min(C_\gamma+\delta\varphi_R,C)$ on $\Lambda$. On $O$ this gives $C_O=C_\gamma+\delta\varphi_R\ge g+\delta$, since $\varphi_R\ge1$.
\end{proof}

\begin{lemma}[Fattening]\label{lem:fatten}
Let $x\in A$, $m:=C(x)-g(x)$ and $\delta\in(0,m]$. There is $\rho>0$ with $B(x;\rho)\subset A$ such that, for every $O\in\Gamma_A$ with $C_O\in\cH(C)$ and $C_O\ge g+\delta$ on $O$, the set $\tilde O:=O\cup B^\circ(x;\rho)$ satisfies
\[
\tilde O\in\Gamma_A,\qquad C_{\tilde O}\in\cH(C),\qquad C_{\tilde O}\le C_O\ \text{ on }\overline\Lambda.
\]
\end{lemma}

\begin{proof}
By continuity and \ref{P:small}, choose $\rho'\in(0,1]$ such that $B(x;2\rho')\subset\Lambda$, $C-g\ge m/2$ on $B(x;2\rho')$, $\omega_g(3\rho')\le\delta/8$, and $(2\norm{g}+\norm{C})\,e_{2\rho'}\le\delta/8$. For $0<\rho<\rho'/2$ let $T_\rho:=\tau_{B(x;\rho)}$ and $p_\rho(w):=\E^w[e^{-\beta T_\rho};\,T_\rho<\tau_{\pd\Lambda}]$.

We claim that $\bar p_\rho:=\sup\{p_\rho(w):w\in\overline\Lambda,\ |w-x|\ge\rho'\}\to0$ as $\rho\downarrow0$. Let $S:=\pd B(x;\rho'/2)\subset\Lambda$. A path started at $w$ with $|w-x|\ge\rho'$ hits $S$ before $B(x;\rho)$, so the strong Markov property gives $p_\rho(w)\le\max_Sp_\rho$. By the strong Markov property $p_\rho$ is $\beta$-harmonic in $\Lambda\setminus B(x;\rho)$, hence continuous on $S$ by Lemma~\ref{lem:interior}, and it decreases as $\rho\downarrow0$. It tends to $0$ pointwise on $S$. On the event that $T_\rho<\tau_{\pd\Lambda}$ for every $\rho$, the times $T_\rho$ increase to a limit $T_*$. If $T_*<\infty$, then $Y_{T_*}=x$ by continuity of paths, and $T_*<\tau_{\pd\Lambda}$: in setting (a) because $T_*\le\tau_{\pd\Lambda}<\infty$ and $x\notin\pd\Lambda$, and in setting (b) trivially; this event is null by \ref{P:polar}. If $T_*=\infty$, which is possible only in setting (b), then $e^{-\beta T_\rho}\to0$ if $\beta>0$; and if $\beta=0$, then by Lemma~\ref{lem:transience} there is almost surely a time after which $|Y_t-x|\ge1$, which bounds $T_\rho$ for $\rho<1$, so this event is null. Bounded convergence gives $p_\rho\to0$ pointwise, and Dini's theorem on the compact set $S$ proves the claim. Fix $\rho\in(0,\rho'/2)$ with $2\norm{C}\bar p_\rho\le\delta/8$.

Then $B(x;\rho)\subset A$, so $\tilde O\in\Gamma_A$, and $C_{\tilde O}\le C_O$ by Lemma~\ref{lem:monotone}(i). By Lemma~\ref{lem:monotone}(ii) it remains to show $C_{\tilde O}\ge g$ on $\tilde O$.

\emph{Far points.} Let $w\in\tilde O$ with $|w-x|\ge\rho'$. Then $w\in O$, and $\tilde O\setminus O\subset B(x;\rho)$, so the hitting time $T$ in Lemma~\ref{lem:monotone}(i) satisfies $T\ge T_\rho$. Hence, by the slack $C_O\ge g+\delta$ on $O$ and the choice of $\rho$,
\begin{equation}\label{eq:far}
C_{\tilde O}(w)\ge C_O(w)-2\norm{C}\,p_\rho(w)\ge g(w)+\delta-\delta/8.
\end{equation}

\emph{Near points.} Let $z\in\tilde O$ with $|z-x|<\rho'$, and under $\PP^z$ let $\tilde\tau:=\tau^+_{\tilde O}$ and $\sigma:=\tau^+_{B^\circ(x;2\rho')}$, both positive. By the strong Markov property at $\tilde\tau\wedge\sigma$, $C_{\tilde O}(z)=\E^z[e^{-\beta(\tilde\tau\wedge\sigma)}\Theta]$, where $\Theta:=C(Y_{\tilde\tau})$ on $\{\tilde\tau\le\sigma\}$ and $\Theta:=C_{\tilde O}(Y_\sigma)$ on $\{\sigma<\tilde\tau\}$. On $\{\tilde\tau\le\sigma\}$, $Y_{\tilde\tau}\in B(x;2\rho')$ and $|Y_{\tilde\tau}-z|\le3\rho'$, so $\Theta\ge g(Y_{\tilde\tau})+m/2\ge g(z)+\delta/2-\delta/8$, using $m\ge\delta$. On $\{\sigma<\tilde\tau\}$, $Y_\sigma\in\tilde O$ with $|Y_\sigma-x|=2\rho'$, so \eqref{eq:far} gives $\Theta\ge g(Y_\sigma)+7\delta/8\ge g(z)+3\delta/4$. Thus $\Theta\ge c:=g(z)+3\delta/8$. Since $m=C(x)-g(x)$, we have $|c|\le\norm{g}+m\le2\norm{g}+\norm{C}$; also $1-\E^z[e^{-\beta(\tilde\tau\wedge\sigma)}]\le\beta\E^z[\sigma]\le e_{2\rho'}$. Hence
\[
C_{\tilde O}(z)\ge c\,\E^z\bigl[e^{-\beta(\tilde\tau\wedge\sigma)}\bigr]\ge c-|c|\,e_{2\rho'}\ge g(z)+3\delta/8-\delta/8>g(z).
\]
\end{proof}

\subsection{Continuity of envelopes}\label{sec:continuity}

The next result combines Lemmas~\ref{lem:truncation}, \ref{lem:slack} and~\ref{lem:fatten} with the interior estimates of Lemma~\ref{lem:interior} to show that envelopes are continuous, through a local estimate at each point of the non-contact set.

\begin{proposition}\label{prop:contenv}
Let $C$ be a ceiling and $W:=\Env[C]$. Then $W$ is continuous on $\overline\Lambda$, and hence a ceiling. More precisely, let $x\in A_C$ and $\delta_1>0$, and let $R>|x|+1$ be such that $\Lambda\subset B^\circ(0;R)$ in setting (a), and $\sup_{|w|\ge R}C(w)\le\delta_1$ in setting (b). Let $\delta_2\in(0,C(x)-g(x)]$, and let $\rho$ be as in Lemma~\ref{lem:fatten} with $\delta=\delta_2$. Then
\[
|W(x)-W(y)|\le\delta_1+\delta_2\bar\varphi_R+\norm{C}\,\Theta_{x,\rho}(|x-y|),\qquad y\in B(x;\rho/4).
\]
\end{proposition}

\begin{proof}
Let $C_\gamma\in\cH(C)$. In setting (b), replace $\gamma$ by $\gamma_R$, which by Lemma~\ref{lem:truncation} is admissible and raises the replacement by at most $\delta_1$; in setting (a), $\gamma\subset B^\circ(0;R)$ already. Lemma~\ref{lem:slack} with $\delta=\delta_2$ then gives $O$ with slack $\delta_2$ and $C_O\le C_\gamma+\delta_1+\delta_2\varphi_R$ on $\Lambda$, and Lemma~\ref{lem:fatten} gives $C_{\tilde O}\in\cH(C)$ with $C_{\tilde O}\le C_O$ and $B^\circ(x;\rho)\subset\tilde O$. By Lemma~\ref{lem:mart}, $C_{\tilde O}$ is $\beta$-harmonic in $B^\circ(x;\rho)$, and $|C_{\tilde O}|\le\norm{C}$. Since $x$ and $y$ lie in $B(0;R)$, where $\varphi_R\le\bar\varphi_R$, Lemma~\ref{lem:interior} gives
\begin{align*}
W(x)\le C_{\tilde O}(x)&\le C_{\tilde O}(y)+\norm{C}\,\Theta_{x,\rho}(|x-y|)\\
&\le C_\gamma(y)+\delta_1+\delta_2\bar\varphi_R+\norm{C}\,\Theta_{x,\rho}(|x-y|),
\end{align*}
and likewise $W(y)\le C_\gamma(x)+\delta_1+\delta_2\bar\varphi_R+\norm{C}\,\Theta_{x,\rho}(|x-y|)$. Taking the infimum over $\cH(C)$ in each inequality gives the bound. Hence $\limsup_{y\to x}|W(x)-W(y)|\le\delta_1+\delta_2\bar\varphi_R$. Here $R$ depends on $\delta_1$ but not on $\delta_2$, so letting $\delta_2\downarrow0$ with $\delta_1$ fixed, and then $\delta_1\downarrow0$, gives continuity at $x$. At $x\in\overline\Lambda\setminus A_C$, Lemma~\ref{lem:onestep}(ii) gives $g\le W\le C$ with equality at $x$, and $g$ and $C$ are continuous. By Lemma~\ref{lem:onestep}(ii), $W$ is a ceiling.
\end{proof}

Continuity of the limit envelope rests on the following inequality, which is obtained from the comparison lemma (Lemma~\ref{lem:comparison}). Its proof takes the Poisson modification of the ceiling on a small ball, and lifts it just enough to clear the gain; the result bounds the envelope above on the ball, and the lift vanishes as the radius shrinks.

\begin{lemma}[Lifted Poisson inequality]\label{lem:lifted}
Let $C$ be a ceiling, let $B:=B^\circ(z_0;r)$ with $B(z_0;r)\subset\Lambda$, and let $R>0$ with $B(z_0;r)\subset B(0;R)$. Then
\[
\Env[C]\le h^B_C+\varepsilon_r\varphi_R\quad\text{on }B,\qquad\varepsilon_r:=\omega_g(2r)+\bigl(\norm{g}+\omega_g(2r)\bigr)e_r.
\]
\end{lemma}

\begin{proof}
Apply Lemma~\ref{lem:comparison} with $U:=B$, and $F:=C+\varepsilon_r\varphi_R$ on $\overline\Lambda$. Since $\varphi_R$ is $\beta$-harmonic near $B(0;R)$, $u=h^B_C+\varepsilon_r\varphi_R$ on $B$. For $z\in B$, the exit position is within $2r$ of $z$, and $1-\E^z[e^{-\beta\tau^+_B}]\le e_r$ as in the proof of Lemma~\ref{lem:fatten}, so
\[
h^B_C(z)\ge h^B_g(z)\ge\bigl(g(z)-\omega_g(2r)\bigr)\E^z\bigl[e^{-\beta\tau^+_B}\bigr]\ge g(z)-\varepsilon_r,
\]
and hypothesis (H1) holds since $\varphi_R\ge1$. Hypothesis (H2) holds since $\varphi_R\ge0$. Hence $\Env[C]\le u$ on $\Lambda$.
\end{proof}

\subsection{Proofs of the main results}\label{sec:proofs}

The next result shows that the iteration is well defined, each envelope being a ceiling, and that it converges uniformly to a limiting ceiling.

\begin{proposition}\label{prop:envelopes}
Let Assumptions~\ref{ass:breg} and~\ref{ass:process} hold, and let $w_0$ be a ceiling. Then each $w_n$ is a ceiling with $V\le w_n\le w_{n-1}$. The limit $w_\infty:=\lim_nw_n$ exists, is a ceiling with $w_\infty\ge V$, and $w_n\to w_\infty$ uniformly on $\overline\Lambda$.
\end{proposition}

\begin{proof}
By induction, Proposition~\ref{prop:contenv} and Lemma~\ref{lem:onestep}(ii), each $w_n$ is a ceiling with $V\le w_n\le w_{n-1}$. Hence $(w_n)$ decreases pointwise to a limit $w_\infty\ge V$, and bounded convergence shows that $w_\infty$ inherits the supermartingale property. It is continuous on $\Lambda$. Indeed, let $z_0\in\Lambda$, let $0<r\le1$ with $B(z_0;r)\subset\Lambda$, let $R\ge|z_0|+1$, and put $B:=B^\circ(z_0;r)$. Lemma~\ref{lem:lifted} with $C=w_n$ gives $w_\infty\le w_{n+1}\le h^B_{w_n}+\varepsilon_r\bar\varphi_R$ on $B$. Since $h^B_{w_n}\le w_n$ on $B$ for every $n$ (Lemma~\ref{lem:monotone}(i)), letting $n\to\infty$ by monotone convergence gives
\[
h^B_{w_\infty}\le w_\infty\le h^B_{w_\infty}+\varepsilon_r\bar\varphi_R\qquad\text{on }B.
\]
Since $h^B_{w_\infty}$ is continuous in $B$ by Lemma~\ref{lem:interior}, the oscillation of $w_\infty$ at $z_0$ is at most $\varepsilon_r\bar\varphi_R$, which tends to $0$ as $r\downarrow0$ with $R$ fixed. In setting (a), $g\le w_\infty\le w_0$ with $g=w_0$ on $\pd\Lambda$, so $w_\infty$ is also continuous at $\pd\Lambda$. Thus $w_\infty$ is continuous on $\overline\Lambda$; it majorises $g$, equals $g$ on $\pd\Lambda$ in setting (a), and vanishes at infinity in setting (b), as $0\le w_\infty\le w_0$. So $w_\infty$ is a ceiling. Finally, Dini's theorem gives uniform convergence on compact subsets of $\overline\Lambda$, which in setting (b) extends to $\R^d$ since $0\le w_n-w_\infty\le w_0\to0$ at infinity.
\end{proof}

To identify the limit with the value function we use a verification result: a ceiling which is $\beta$-harmonic in its non-contact set is the value function, and the first exit from that set is optimal.

\begin{lemma}[Verification]\label{lem:verification}
Let $u$ be a ceiling which is $\beta$-harmonic in $U:=A_u=\{x\in\Lambda : u(x)>g(x)\}$. Then $u=V$. Moreover, for $x\in\overline\Lambda$ let $U^x$ be the connected component of $U$ containing $x$ (with $\pd U^x:=\{x\}$ if $x\notin U$), and let $\sigma$ be the first hitting time of $\pd U^x$ by $X^x$. Then $\sigma$ is optimal: $V(x)=\E[e^{-\beta\sigma}g(X^x_\sigma)\1_{\{\sigma<\infty\}}]$.
\end{lemma}

\begin{proof}
Since $u$ is a ceiling, $u\ge V$ (Definition~\ref{def:ceiling}). If $x\notin U$ then $\sigma=0$ and $u(x)=g(x)\le V(x)$. Let $x\in U$. As $\sigma$ is the first hitting time of the closed set $\overline\Lambda\setminus U^x\supset\pd\Lambda$, it is admissible by Lemma~\ref{lem:admissible}. Let $U_k$ be open sets with $\overline{U_k}\subset U^x$ compact and $U_k\uparrow U^x$; by continuity of paths the limit of the increasing exit times $\sigma_k$ of $U_k$ is $\sigma$. By Lemma~\ref{lem:mart}, $u(x)=\E[e^{-\beta(t\wedge\sigma_k)}u(X^x_{t\wedge\sigma_k})]$ for all $t\ge0$ and $k$. Letting $k\to\infty$, continuity of $u$ and bounded convergence give $u(x)=\E[e^{-\beta(t\wedge\sigma)}u(X^x_{t\wedge\sigma})]$. Now let $t\to\infty$. On $\{\sigma<\infty\}$ the integrand converges to $e^{-\beta\sigma}u(X^x_\sigma)$. On $\{\sigma=\infty\}$, which is null in setting (a), it converges to zero: by discounting if $\beta>0$, and, if $\beta=0$, because $|X^x_t|\to\infty$ by Lemma~\ref{lem:transience} and $u$ vanishes at infinity. By bounded convergence, $u(x)=\E[e^{-\beta\sigma}u(X^x_\sigma)\1_{\{\sigma<\infty\}}]$. Points of $\pd U^x\cap\Lambda$ do not belong to $U$, so on $\{\sigma<\infty\}$ we have $X^x_\sigma\in(\Lambda\setminus U)\cup\pd\Lambda$, where $u=g$. Hence $u(x)=\E[e^{-\beta\sigma}g(X^x_\sigma)\1_{\{\sigma<\infty\}}]\le V(x)$. Thus $u=V$, and $\sigma$ is optimal.
\end{proof}

By Lemma~\ref{lem:verification}, it remains to show that the limit envelope is harmonic off its contact set, which we do as in Perron's method. On a small ball inside $A_\infty$, the Poisson modification of $w_\infty$ clears $g$, so the comparison lemma, applied with ceiling $w_N$ for large $N$, bounds $w_{N+1}$, and hence $w_\infty$, above by it, up to an arbitrarily small lift; superharmonicity gives the reverse inequality.

\begin{lemma}\label{lem:harmonic}
Under the hypotheses of Proposition~\ref{prop:envelopes}, the limit envelope $w_\infty$ is $\beta$-harmonic in the open set $A_\infty$.
\end{lemma}

\begin{proof}
The set $A_\infty$ is open since $w_\infty$ and $g$ are continuous. Fix $x_0\in A_\infty$ and set $m:=w_\infty(x_0)-g(x_0)>0$. Fix $R\ge|x_0|+2$, with $\Lambda\subset B^\circ(0;R)$ in setting (a). For a closed ball $B\subset\Lambda$ and $z\in B$, write $\sigma^z_B$ for the first hitting time of $\pd B$ by $X^z$; since $1-e^{-s}\le s$, by \ref{P:small}
\begin{equation}\label{eq:smallball}
\sup_{z\in B}\bigl(1-\E[e^{-\beta\sigma^z_B}]\bigr)\le\beta\sup_{z\in B}\E[\sigma^z_B]\to0\qquad\text{as the radius of $B$ tends to zero.}
\end{equation}
By continuity and \eqref{eq:smallball}, there is a closed ball $B=B(x_0;\varrho)\subset\Lambda$ with $\varrho\le1$ such that $|w_\infty-w_\infty(x_0)|<m/3$, $|g-g(x_0)|<m/3$ on $B$, and $\norm{w_\infty}_\infty\bigl(1-\E[e^{-\beta\sigma^z_B}]\bigr)<m/3$ for $z\in B$. Then $w_\infty>g+m/3$ on $B$, so $B\subset A_\infty\subset A_N$ for every $N$. Write $B^\circ:=B^\circ(x_0;\varrho)$ and let $P:=h^{B^\circ}_{w_\infty}$, which is $\beta$-harmonic in $B^\circ$ by Lemma~\ref{lem:mart}. By the supermartingale property, $P\le w_\infty$ on $B^\circ$. Since $\bigl|P(z)-\E[w_\infty(X^z_{\sigma^z_B})]\bigr|\le\norm{w_\infty}_\infty\bigl(1-\E[e^{-\beta\sigma^z_B}]\bigr)$, for $z\in B^\circ$ we have
\[
P(z)\ \ge\ \E\bigl[w_\infty(X^z_{\sigma^z_B})\bigr]-\tfrac{m}{3}\ >\ w_\infty(x_0)-\tfrac{2m}{3}\ =\ g(x_0)+\tfrac m3\ >\ g(z).
\]

Let $\varepsilon>0$. By Proposition~\ref{prop:envelopes} there exists $N$ with $w_N\le w_\infty+\varepsilon$ on $\overline\Lambda$. Apply Lemma~\ref{lem:comparison} with ceiling $C=w_N$, with $U=B^\circ$, and with $F:=w_\infty+\varepsilon\varphi_R$ on $\overline\Lambda$. Since $\varphi_R$ is $\beta$-harmonic near $B$, $u=h^U_F=P+\varepsilon\varphi_R$ on $B^\circ$. Hypothesis (H1) of Lemma~\ref{lem:comparison} holds since $u\ge P>g$ in $U$, and (H2) holds since $F\ge w_\infty+\varepsilon\ge w_N$ on $\overline\Lambda$, as $\varphi_R\ge1$. Hence
\[
w_\infty\le w_{N+1}=\Env[w_N]\le P+\varepsilon\varphi_R\le P+\varepsilon\bar\varphi_R\qquad\text{on } B^\circ.
\]
Letting $\varepsilon\downarrow0$ gives $w_\infty\le P$, and therefore $w_\infty=P=h^{B^\circ}_{w_\infty}$ on $B^\circ$. As $x_0\in A_\infty$ was arbitrary, Lemma~\ref{lem:local} shows that $w_\infty$ is $\beta$-harmonic in $A_\infty$.
\end{proof}

\begin{proof}[Proof of Theorem~\ref{thm:main}]
(i) By Proposition~\ref{prop:envelopes}, each $w_n$ is a ceiling and $V\le w_{n+1}\le w_n$. Since $w_n$ is continuous, $A_n$ is open, so $\tau_n$, which is zero for $x\notin A_n$ and otherwise the first hitting time of the closed set $\overline\Lambda\setminus A^x_n\supset\pd\Lambda$, is admissible by Lemma~\ref{lem:admissible}, and $\ow_{n+1}\le V$ by definition of $V$. Since $w_{n+1}\le w_n$, $A_{n+1}\subset A_n$. Fix $x\in A_{n+1}$. Up to $\tau_{n+1}$ the path stays in $A^x_{n+1}\subset A^x_n$, so $\tau_{n+1}\le\tau_n$ and, by the strong Markov property at $\tau_{n+1}$, $\ow_n(x)=\E[e^{-\beta\tau_{n+1}}\ow_n(X^x_{\tau_{n+1}})\1_{\{\tau_{n+1}<\infty\}}]$. At $y:=X^x_{\tau_{n+1}}$ we have $\ow_n(y)\le V(y)\le w_{n+1}(y)=g(y)$, so $\ow_n(x)\le\ow_{n+1}(x)$. For $x\notin A_{n+1}$, $\ow_n(x)\le V(x)\le w_{n+1}(x)=g(x)=\ow_{n+1}(x)$.

(ii) By Proposition~\ref{prop:envelopes} and Lemma~\ref{lem:harmonic}, $w_\infty$ is a ceiling which is $\beta$-harmonic in $A_\infty=A_{w_\infty}$, so Lemma~\ref{lem:verification} gives $w_\infty=V$. Together with Proposition~\ref{prop:envelopes} this proves the uniform convergence, and also $A_\infty=\{V>g\}$. Lemma~\ref{lem:verification} also shows that the first hitting time of $\pd A^x_\infty$, namely $\tau_\infty$, is optimal, so that $\ow_\infty(x)=\E[e^{-\beta\tau_\infty}g(X^x_{\tau_\infty})\1_{\{\tau_\infty<\infty\}}]=V(x)$ for every $x\in\overline\Lambda$.

(iii) The rule $\sigma_n$ is the first hitting time of the closed set $\{w_n\le g+\varepsilon\}$, which in setting (a) contains $\pd\Lambda$ since $w_n=g$ there, so $\sigma_n$ is admissible by Lemma~\ref{lem:admissible}. Hence $v_n\le V$ by definition of $V$, and $V\le w_n$ by (i). By (ii) there is $N$ such that $\norm{w_n-V}_\infty<\varepsilon$ for all $n\ge N$. Fix such an $n$ and $x\in\overline\Lambda$, and write $\delta:=\norm{w_n-V}_\infty$ and $\sigma:=\sigma_n$. It suffices to show that $v_n(x)\ge V(x)-\varepsilon$. If $\sigma=0$ then $V(x)\le w_n(x)\le g(x)+\varepsilon=v_n(x)+\varepsilon$, so suppose that $\sigma>0$, and let $U^x$ be the connected component of $\{V>g\}$ containing $x$. For $t<\sigma$ we have $V(X^x_t)\ge w_n(X^x_t)-\delta>g(X^x_t)+\varepsilon-\delta>g(X^x_t)$, so $X^x_t\in U^x$; hence $\sigma$ is at most the exit time of $U^x$. By (ii) and Lemma~\ref{lem:harmonic}, $V=w_\infty$ is $\beta$-harmonic in $U^x$. Arguing as in the proof of Lemma~\ref{lem:verification}, with the exit times there replaced by their minima with $\sigma$, gives $V(x)=\E[e^{-\beta\sigma}V(X^x_\sigma)\1_{\{\sigma<\infty\}}]$. On $\{\sigma<\infty\}$, continuity gives $w_n(X^x_\sigma)\le g(X^x_\sigma)+\varepsilon$, so
\[
V(x)\le\E\bigl[e^{-\beta\sigma}w_n(X^x_\sigma)\1_{\{\sigma<\infty\}}\bigr]\le v_n(x)+\varepsilon.\qedhere
\]
\end{proof}

\section{Translation invariance and the modulus of continuity}\label{sec:translation}

In this section the process is Brownian motion: $\PP^z$ is the law of $z+B$. We prove Theorem~\ref{thm:modulus}. For $x\in\R^d$ let
\begin{equation}\label{eq:corrector}
\varphi_x(z):=\cosh\bigl(\sqrt{2\beta}\,(z-x)\cdot\nu\bigr),\qquad z\in\R^d,
\end{equation}
where $\nu$ is a fixed unit vector. We call $\varphi_x$ the \emph{corrector} at $x$, as in \ref{P:corr}: it plays the role of the constant function $1$, which is not $\beta$-harmonic when $\beta>0$. It is $\beta$-harmonic in $\R^d$ and satisfies $\varphi_x\ge1$ and $\varphi_x(x)=1$; for $\beta=0$, $\varphi_x\equiv1$. As in Section~\ref{sec:esp}, $C$ is a ceiling and $W:=\Env[C]$.

\begin{lemma}[Translation]\label{lem:translation}
For all $x,y\in\overline\Lambda$,
\[
W(x)\le W(y)+\max(\omega_C,\omega_g)(|x-y|).
\]
In particular $\omega_W\le\max(\omega_C,\omega_g)$.
\end{lemma}

\begin{proof}
Let $v:=x-y$ and $c:=\max(\omega_C,\omega_g)(|v|)$, and fix $C_\gamma\in\cH(C)$; it suffices to show that $W(x)\le C_\gamma(y)+c$. If $x\in\pd\Lambda$ in setting (a), then $W(x)=g(x)\le g(y)+c\le C_\gamma(y)+c$ by Lemma~\ref{lem:onestep}. If $\gamma=\emptyset$, then $W(x)\le C(x)\le C(y)+c$. Otherwise let $x\in\Lambda$, and apply Lemma~\ref{lem:comparison} with $U:=\gamma+v$, which is bounded and open, and
\[
F(z):=\begin{cases}C(z-v)+c\,\varphi_x(z)&\text{if } z-v\in\overline\Lambda,\\ \max C+c\,\varphi_x(z)&\text{otherwise}.\end{cases}
\]
Then $F$ is Borel and bounded on bounded sets. For $z$ with $z-v\in\overline\Lambda$, translating by $-v$ maps the first exit of $z+B$ from $U$ to the first exit of $(z-v)+B$ from $\gamma$, at a point of $\overline\Lambda$; and $e^{-\beta t}\varphi_x(z+B_t)$ is a martingale, bounded up to the exit from the bounded set $U$. Hence
\[
u(z)=h^U_F(z)=C_\gamma(z-v)+c\,\varphi_x(z)\qquad\text{whenever } z-v\in\overline\Lambda.
\]
Hypothesis (H1) of Lemma~\ref{lem:comparison} holds: for $z\in U\cap\overline\Lambda$ we have $z-v\in\gamma$, so $u(z)\ge g(z-v)+c\ge g(z)$, using $\varphi_x\ge1$. Hypothesis (H2) holds: for $z\in\overline\Lambda\setminus U$, $F(z)\ge C(z-v)+c\ge C(z)$ if $z-v\in\overline\Lambda$, and $F(z)\ge\max C\ge C(z)$ otherwise. Lemma~\ref{lem:comparison} gives $W(x)\le u(x)=C_\gamma(y)+c$, since $\varphi_x(x)=1$.
\end{proof}

\begin{proof}[Proof of Theorem~\ref{thm:modulus}]
In setting (b) with $\beta=0$ and $d=2$, a ceiling exists only if $g\equiv0$ (Remark~\ref{rem:recurrent}), and then $w_n\equiv V\equiv0$. In all other cases Assumption~\ref{ass:process} holds, and by Theorem~\ref{thm:main} each $w_n$ is a ceiling. By Lemma~\ref{lem:translation}, $\omega_{w_n}\le\max(\omega_{w_{n-1}},\omega_g)$ for every $n\ge1$, so $\omega_{w_n}\le\omega$ by induction; this passes to the pointwise limit $V=\lim_nw_n$.
\end{proof}

The proof of Lemma~\ref{lem:translation} uses only that the law of $Y-Y_0$ under $\PP^z$ does not depend on $z$, together with, for each $x$, a $\beta$-harmonic function $\varphi_x\ge1$ in $\R^d$ with $\varphi_x(x)=1$. It therefore applies to diffusions with constant coefficients satisfying Assumption~\ref{ass:process}, with $\varphi_x$ built from exponential solutions of $L\varphi=\beta\varphi$ along a fixed direction, where $L$ is the generator.

\appendix
\section{Proofs of standard facts}\label{app:conseq}

\begin{proof}[Proof of Lemma~\ref{lem:mart}]
Let $\tau:=\tau^+_G$, which is finite almost surely by \ref{P:small}, and $N_t:=\E^z[e^{-\beta\tau}u(Y_\tau)\,|\,\mathcal F_t]$, a martingale bounded by $\sup_{\overline G}|u|$. On $\{\tau\le t\}$, $N_t=e^{-\beta\tau}u(Y_\tau)$. On $\{t<\tau\}$ we have $Y_t\in G$, and the remaining exit time is $\tau^+_G$ for the path restarted at time $t$, so the Markov property and $u=h^G_u$ on $G$ give $N_t=e^{-\beta t}h^G_u(Y_t)=e^{-\beta t}u(Y_t)$. Hence $N_t=e^{-\beta(t\wedge\tau)}u(Y_{t\wedge\tau})$. No differentiability of $u$, and so no It\^o formula, is involved: $\beta$-harmonicity is defined through exit distributions. Measurability of $h^U_F$ follows from measurability of $z\mapsto\PP^z$. If $G$ has compact closure in $U$ and $z\in G$, then at time $\tau^+_G$ the path is at a point of $\pd G\subset U$, so the strong Markov property gives $h^U_F(z)=h^G_{h^U_F}(z)$.
\end{proof}

\begin{proof}[Proof of Lemma~\ref{lem:local}]
If $B$ is an open ball with $u=h^B_u$ on $B$, and $G'$ is open with closure in $B$, then $u=h^{G'}_u$ on $G'$, by the strong Markov property at the exit from $G'$, which occurs at a point of $B$. Now let $G$ be open with compact closure $K\subset O$, and $z\in G$. Cover $K$ by finitely many balls $B_{z_i}=B^\circ(z_i;r_i)$ with $z_i\in K$ and $K\subset\bigcup_iB^\circ(z_i;r_i/2)$, and let $0<r\le\min_ir_i/4$. For $y\in K$, the ball $B^\circ(y;2r)$ lies in some $B_{z_i}$, so by the first observation and Lemma~\ref{lem:mart}, $e^{-\beta(t\wedge\sigma)}u(Y_{t\wedge\sigma})$ is a bounded $\PP^y$-martingale, where $\sigma$ is the exit time from $B^\circ(y;r)$. Let $\tau:=\tau^+_G$, $\sigma_0:=0$ and $\sigma_{k+1}:=\inf\{t>\sigma_k:|Y_t-Y_{\sigma_k}|\ge r\}\wedge\tau$. By the strong Markov property at $\sigma_k$ and optional sampling, $u(z)=\E^z[e^{-\beta\sigma_k}u(Y_{\sigma_k})]$ for every $k$. Since $\tau<\infty$ almost surely by \ref{P:small}, and paths are uniformly continuous on $[0,\tau]$, almost surely $\sigma_k=\tau$ for all large $k$. As $u$ is bounded on the $r$-neighbourhood of $K$, bounded convergence gives $u(z)=\E^z[e^{-\beta\tau}u(Y_\tau)]=h^G_u(z)$.
\end{proof}

\begin{proof}[Proof of Lemma~\ref{lem:interior}]
Let $G:=B^\circ(z;r/2)$ and, for $w\in B(z;r/4)$, let $\sigma:=\tau^+_G$ under $\PP^w$, so that $\{\sigma\le\epsilon\}\subset\{\sup_{t\le\epsilon}|Y_t-w|\ge r/4\}$. Since $u=h^G_u$ on $G$, the Markov property at time $\epsilon$ gives $\E^w[e^{-\beta\sigma}u(Y_\sigma)\,|\,\mathcal F_\epsilon]=e^{-\beta\epsilon}u(Y_\epsilon)$ on $\{\epsilon<\sigma\}$, where $Y_\epsilon\in G$. With $\tilde u:=u\1_G$, it follows that $|u(w)-e^{-\beta\epsilon}\E^w[\tilde u(Y_\epsilon)]|\le2M\,\PP^w(\sigma\le\epsilon)\le2M\,\delta_{r/4}(\epsilon)$. Applying this at $w=y$ and $w=z$, and using $|\tilde u|\le M$, gives the bound. For the second statement put $\Theta_{z,r}(s):=\inf_{\epsilon>0}\bigl(4\delta_{r/4}(\epsilon)+\sup_{|y-z|\le s}\|P_\epsilon(y,\cdot)-P_\epsilon(z,\cdot)\|_{TV}\bigr)$ for $s\le r/4$, and $\Theta_{z,r}(s):=2$ for $s>r/4$; it tends to $0$ by \ref{P:small} and \ref{P:tv}. The last claim follows from Lemma~\ref{lem:mart}.
\end{proof}

\begin{proof}[Proof of Lemma~\ref{lem:reg}]
(i) Let $0<s<\delta$. By the Markov property at time $s$, $\PP^z(\tau^+_D>\delta)\le\E^z[\psi(Y_s)]$, where $\psi(w):=\PP^w(Y_t\in D\text{ for all }t\in[0,\delta-s])$. The right-hand side is continuous in $z$ by \ref{P:tv}, and at $z=y$ it equals $\PP^y(Y_t\in D\text{ for all }t\in[s,\delta])$. As $s\downarrow0$ this decreases to $\PP^y(Y_t\in D\text{ for all }t\in(0,\delta])$, which vanishes if $y$ is regular. If $y$ is not regular, then by \ref{P:markov} there is $\delta>0$ such that $p:=\PP^y(Y_t\in D\text{ for all }t\in(0,2\delta])>0$. For $s\in(0,\delta)$, the Markov property at time $s$ gives $p\le\E^y[\phi(Y_s);\,Y_t\in D\text{ for all }t\in(0,s]]$, where $\phi(w):=\PP^w(\tau^+_D>\delta)$. If $\phi(w)\to0$ as $w\to y$ in $D$, the right-hand side tends to $0$ as $s\downarrow0$ by dominated convergence, which contradicts $p>0$.

(ii) Let $\eta,s>0$. On the event $\{\tau^+_D\le s,\ \sup_{t\le s}|Y_t-z|\le\eta\}$ we have $|Y_{\tau^+_D}-y|\le\eta+|z-y|$ and $1-e^{-\beta\tau^+_D}\le1-e^{-\beta s}$, and the complementary event has $\PP^z$-probability at most $\PP^z(\tau^+_D>s)+\delta_\eta(s)$. Since $F$ is bounded on the bounded set $\overline D$, letting $z\to y$ by (i), then $s\downarrow0$, then $\eta\downarrow0$, gives $h^D_F(z)\to F(y)$.
\end{proof}

\begin{proof}[Proof of Lemma~\ref{lem:admissible}]
(i) In setting (a), $\tau_{\pd\Lambda}<\infty$ almost surely by \ref{P:small}, since $\Lambda$ is bounded, and $X^x_{\tau_{\pd\Lambda}}\in\pd\Lambda\subset E$, so the first hitting time of $E$ is at most $\tau_{\pd\Lambda}$.

(ii) If $x\in\Lambda$, then $Y_{\tau_{\pd\Lambda}}\in\pd\Lambda$ lies outside $O$, so $\tau^+_O\le\tau_{\pd\Lambda}$. If $x\in\pd\Lambda$, then $\tau_{\pd\Lambda}=0$, and since $x$ is regular the path leaves $\Lambda$, and hence $O$, at arbitrarily small positive times, so $\tau^+_O=0$. For $\vartheta'$, on $\{\sigma<\infty\}$ the process $(Y_{\sigma+s})_{s\ge0}$ has conditional law $\PP^{Y_\sigma}$ given $\mathcal F_\sigma$, by the strong Markov property, and $\vartheta'-\sigma$ is its time $\tau^+_O$, while its first hitting time of $\pd\Lambda$ is $\tau_{\pd\Lambda}-\sigma$, since $\sigma\le\tau_{\pd\Lambda}$. The first part, applied under $\PP^{Y_\sigma}$, gives $\vartheta'\le\tau_{\pd\Lambda}$.

(iii) This is \ref{P:small}, together with the strong Markov property at $\sigma$ for $\vartheta'$.
\end{proof}

\begin{lemma}[Superharmonic functions]\label{lem:superharm}
Let $C\in C(\overline\Lambda)$ be bounded, and write $\tau:=\tau_{\pd\Lambda}$.
\begin{enumerate}[label=(\roman*)]
\item If $\E[e^{-\beta\sigma}C(X^x_\sigma)\1_{\{\sigma<\infty\}}]\le C(x)$ for every $x\in\overline\Lambda$ and every admissible stopping time $\sigma$ (it suffices to take $\sigma=r\wedge\tau$, $r\ge0$), then $C$ is $\beta$-superharmonic in $\Lambda$.
\item Conversely, if $C$ is $\beta$-superharmonic, this inequality holds for every admissible $\sigma$ in setting (a), and in setting (b) if $C\ge0$.
\item If $u\in C(\overline\Lambda)$ is bounded and $\beta$-harmonic in $\Lambda$, then $e^{-\beta(t\wedge\tau)}u(X^x_t)$ is a martingale for every $x\in\overline\Lambda$. In particular, $u$ and $-u$ are $\beta$-superharmonic.
\end{enumerate}
\end{lemma}

\begin{proof}
(i) Let $M_t:=e^{-\beta(t\wedge\tau)}C(X^x_t)$ and $s\le t$. On $\{\tau\le s\}$, $M_t=M_s$. On $\{s<\tau\}$, the Markov property at time $s$ shows that $\E[M_t\,|\,\mathcal F_s]$ is $e^{-\beta s}$ times $\E^y[e^{-\beta((t-s)\wedge\tau)}C(X^y_{t-s})]$ evaluated at $y=X^x_s$, which is at most $e^{-\beta s}C(X^x_s)=M_s$ by the hypothesis with $\sigma=(t-s)\wedge\tau$.

(ii) For bounded $\sigma$ this is optional sampling. In setting (a), $\sigma\le\tau<\infty$ and $M$ is bounded, so let $n\to\infty$ in $\sigma\wedge n$ by bounded convergence. In setting (b) with $C\ge0$, use Fatou's lemma instead.

(iii) Let $G_k$ be open with compact closures in $\Lambda$ and $G_k\uparrow\Lambda$. By Lemma~\ref{lem:mart}, the processes stopped at the exit times of $G_k$ are bounded martingales. These exit times increase to $\tau$ by continuity of paths, and bounded convergence, with continuity of $u$ on $\overline\Lambda$, passes the martingale property to the limit.
\end{proof}

\begin{remark}[Brownian motion]\label{rem:bmprops}
For Brownian motion, properties \ref{P:markov}--\ref{P:trans} are classical \cite{Karatzas_Shreve,RevuzYor}. \ref{P:markov} is the strong Markov property together with Blumenthal's zero--one law, and \ref{P:cont} follows from martingale representation in the Brownian filtration. In \ref{P:tv}, $P_t(z,\cdot)$ is the Gaussian law $N(z,tI)$, which is continuous in $z$ in total variation by Scheff\'e's lemma. In \ref{P:small}, $\delta_r(\epsilon)\to0$ by scaling, $\E^z[\tau^+_{B^\circ(x;r)}]=(r^2-|z-x|^2)/d\le r^2/d$, and exit times of bounded sets have bounded means. \ref{P:polar} holds because points are polar for $d\ge2$. For \ref{P:corr} take $\varphi_R(z):=\cosh(\sqrt{2\beta}\,z\cdot\nu)$ for every $R$, where $\nu$ is a fixed unit vector; this is $\beta$-harmonic on $\R^d$ and at least $1$. Finally, \ref{P:trans} holds if $\beta>0$, or if $\beta=0$ and $d\ge3$. Indeed, $q_\rho$ of \eqref{eq:qrho} is a nonnegative $\beta$-superharmonic function equal to $1$ on $B(0;\rho)$ (see the proof of Lemma~\ref{lem:ceiling}(ii)), so Lemma~\ref{lem:superharm}(ii), applied to $q_\rho$ at the hitting time of $B(0;\rho)$, shows that the function in \ref{P:trans} is at most $q_\rho$, which vanishes at infinity.
\end{remark}

\begin{lemma}[Expected exit time]\label{lem:exittime}
In setting (a), let $e(x):=\E^x[\tau_{\pd\Lambda}]$ and $\tau:=\tau_{\pd\Lambda}$. Then $e$ is bounded, continuous on $\overline\Lambda$, positive in $\Lambda$ and zero on $\pd\Lambda$, and $e(X^x_t)+t\wedge\tau$ is a martingale for every $x\in\overline\Lambda$.
\end{lemma}

\begin{proof}
Boundedness is \ref{P:small}, $e=0$ on $\pd\Lambda$ by definition, and $e>0$ in $\Lambda$ by continuity of paths. By the Markov property at time $t$, $\E^x[\tau\,|\,\mathcal F_t]=t+e(Y_t)$ on $\{t<\tau\}$ and $=\tau$ on $\{\tau\le t\}$, which is $t\wedge\tau+e(X^x_t)$ since $e=0$ on $\pd\Lambda$; this is the martingale. For continuity in $\Lambda$, fix $x_0\in\Lambda$ and $\eta>0$ with $B(x_0;2\eta)\subset\Lambda$, and extend $e$ by $0$ outside $\overline\Lambda$. For $x\in B(x_0;\eta)$ and $\epsilon>0$, the Markov property at time $\epsilon$ gives $e(x)=\E^x[\tau\wedge\epsilon]+\E^x[e(Y_\epsilon);\,\epsilon<\tau]$, so
\[
\bigl|e(x)-\E^x[e(Y_\epsilon)]\bigr|\le\epsilon+\sup e\cdot\PP^x(\tau\le\epsilon)\le\epsilon+\sup e\cdot\delta_\eta(\epsilon).
\]
The function $x\mapsto\E^x[e(Y_\epsilon)]$ is continuous by \ref{P:tv}, and the right-hand side tends to $0$ as $\epsilon\downarrow0$, uniformly in $x$, so $e$ is continuous on $B(x_0;\eta)$. At $y\in\pd\Lambda$, the same identity with $\epsilon=s$ gives $e(x)\le s+\sup e\cdot\PP^x(\tau>s)$, which tends to $s$ as $x\to y$ by Lemma~\ref{lem:reg}(i), since $y$ is regular; letting $s\downarrow0$ gives $e(x)\to0=e(y)$.
\end{proof}

\begin{lemma}[Transience]\label{lem:transience}
In setting (b) with $\beta=0$, $|Y_t|\to\infty$ $\PP^z$-almost surely, for every $z$.
\end{lemma}

\begin{proof}
Fix $\rho>0$, and for $n>\rho$ let $\zeta_n$ be the exit time from $B^\circ(0;n)$, which is finite almost surely by \ref{P:small}. By the strong Markov property at $\zeta_n$, the probability that $Y$ visits $B(0;\rho)$ after time $\zeta_n$ is at most $\sup_{|w|=n}\PP^w(\tau_{B(0;\rho)}<\infty)$, which tends to $0$ as $n\to\infty$ by \ref{P:trans}. If $Y$ visits $B(0;\rho)$ at arbitrarily large times, it does so after every $\zeta_n$. Hence, almost surely, $Y$ eventually stays outside $B(0;\rho)$; as $\rho$ is arbitrary, $|Y_t|\to\infty$.
\end{proof}

\begin{remark}\label{rem:bmclassical}
For Brownian motion, a continuous function $C$ with $\frac12\Delta C-\beta C\le0$ in the sense of distributions is $\beta$-superharmonic in the above sense, by It\^o's formula applied to mollifications on compact subsets of $\Lambda$, together with continuity of $C$ up to $\pd\Lambda$ in setting (a). Conversely, a continuous function $u$ satisfying $u(x)\ge\E^x[e^{-\beta\sigma}u(Y_\sigma)]$ for the exit times $\sigma$ of all small balls centred at $x$ satisfies $\frac12\Delta u-\beta u\le0$ in the sense of distributions. Indeed, $\sigma$ and $Y_\sigma$ are independent, so the right-hand side is $\E[e^{-\beta\sigma}]$ times a spherical mean; the inequality survives mollification, and for smooth $u$ an expansion to second order gives the claim \cite{ArmitageGardiner}. These facts are used in Remarks~\ref{rem:recurrent} and~\ref{rem:bmprops}, and to verify that the explicit Brownian ceilings of Lemma~\ref{lem:ceiling} are ceilings.
\end{remark}

\begin{proof}[Proof of Lemma~\ref{lem:onestep}]
(i) For $\gamma=\emptyset$ this is the supermartingale property of $C$, so let $\gamma\in\Gamma_A$ be nonempty and $x\in\overline\Lambda$. Work under $\PP^x$, and write $Z:=Y$, $\vartheta:=\tau^+_\gamma$, and $M_t:=e^{-\beta(t\wedge\tau_{\pd\Lambda})}C(X^x_{t\wedge\tau_{\pd\Lambda}})$, which is a supermartingale, nonnegative in setting (b). Since $\gamma\subset\Lambda$ is bounded, $\vartheta$ is finite and admissible by Lemma~\ref{lem:admissible}. Hence the definition of $C_\gamma$ gives $C_\gamma(x)=\E[M_\vartheta]$, and $\E[M_\vartheta]\le C(x)$ by optional sampling. Now let $\sigma$ be admissible and $\vartheta':=\inf\{t>\sigma:Z_t\notin\gamma\}$, so that $\vartheta'=\infty$ on $\{\sigma=\infty\}$. By Lemma~\ref{lem:admissible}, $\vartheta'$ is finite almost surely on $\{\sigma<\infty\}$, and $\vartheta'\le\tau_{\pd\Lambda}$ in setting (a). On $\{\sigma<\infty\}$, $\vartheta'-\sigma$ is the first exit time from $\gamma$, after time $0$, of $(Z_{\sigma+s})_{s\ge0}$, which by the strong Markov property has conditional law $\PP^{Z_\sigma}$ given $\mathcal F_\sigma$, where $Z_\sigma=X^x_\sigma$ as $\sigma\le\tau_{\pd\Lambda}$. Conditioning on $\mathcal F_\sigma$ therefore gives
\[
\E\bigl[e^{-\beta\sigma}C_\gamma(X^x_\sigma)\1_{\{\sigma<\infty\}}\bigr]=\E\bigl[M_{\vartheta'}\1_{\{\sigma<\infty\}}\bigr].
\]
In setting (b) let $M_\infty:=\lim_tM_t=0$, which exists and vanishes since $C$ is bounded and vanishes at infinity, by discounting if $\beta>0$ and by Lemma~\ref{lem:transience} if $\beta=0$; the nonnegative supermartingale $M$ is then closed by $M_\infty$, and $M_{\vartheta'}\1_{\{\sigma<\infty\}}=M_{\vartheta'}$. On $\{\sigma<\vartheta\}$ the path lies in $\gamma$ on $(0,\sigma]$, so $\vartheta'=\vartheta$. On $\{\vartheta\le\sigma\}$, which belongs to $\mathcal F_\vartheta$, we have $\vartheta\le\vartheta'$, and optional sampling gives $\E[M_{\vartheta'};\vartheta\le\sigma]\le\E[M_\vartheta;\vartheta\le\sigma]$. Adding the two parts, the left-hand side above is at most $\E[M_\vartheta]=C_\gamma(x)$. Finally $C_\gamma\ge g$ by admissibility.

(ii) Since $C=C_\emptyset\in\cH(C)$ and every element of $\cH(C)$ majorises $g$, we have $g\le W\le C$. As $C=g$ on $\Lambda\setminus A$ by definition of $A$, and on $\pd\Lambda$ in setting (a) by definition of a ceiling, it follows that $W=g=C$ on $\overline\Lambda\setminus A$. By (i), for every $C_\gamma\in\cH(C)$ and every admissible $\tau$,
\[
\E\bigl[e^{-\beta\tau}g(X^x_\tau)\1_{\{\tau<\infty\}}\bigr]\le\E\bigl[e^{-\beta\tau}C_\gamma(X^x_\tau)\1_{\{\tau<\infty\}}\bigr]\le C_\gamma(x),
\]
whence $C_\gamma\ge V$ and $W\ge V$. If $W$ is continuous, then $W(X^x_\sigma)$ is a random variable, and taking the infimum over $\cH(C)$ in (i) gives the inequality of (i) with $W$ in place of $C_\gamma$, for every admissible $\sigma$, so $W$ is $\beta$-superharmonic in $\Lambda$ by Lemma~\ref{lem:superharm}(i). It majorises $g$; it equals $g$ on $\pd\Lambda$ in setting (a); and in setting (b) it vanishes at infinity, since $0\le g\le W\le C$. So $W$ is a ceiling.
\end{proof}

\section{Ceilings}\label{app:ceilings}

This appendix shows that ceilings exist in both settings. Write $\bar g:=\max g$, and in setting (a) let $e(x):=\E^x[\tau_{\pd\Lambda}]$ be the expected exit time (Lemma~\ref{lem:exittime}).

\begin{lemma}\label{lem:ceiling}
\begin{enumerate}[label=(\roman*)]
\item In setting (a), let $H:=h^\Lambda_{\tilde g}$ be the $\beta$-harmonic extension of $g|_{\pd\Lambda}$, where $\tilde g$ is a bounded continuous extension of $g$ to $\R^d$, let $\bar e:=\max_{\overline\Lambda}e$, and consider the nondecreasing function of one real variable
\[
m(t):=\max\{(g-H)^+(x) : x\in\overline\Lambda,\ e(x)\le t\},\qquad t\in[0,\bar e],
\]
with maximum $M:=m(\bar e)$, and let $\psi$ be the least concave majorant of $m$ on the interval $[0,\bar e]$, again a function of one real variable. Then $H+\psi\circ e$ is a ceiling. If $g\ge0$ and $\supp(g)$ is contained in a nonempty compact set $K\subset\Lambda$, then \eqref{eq:ceile} is also a ceiling; for Brownian motion on the unit ball with $K=B(0;\rho)$, it is \eqref{eq:w0}.
\item For Brownian motion in setting (b), with $\beta>0$ or $d\ge3$, choose $0=\rho_0<\rho_1<\rho_2<\dots\to\infty$ with $g\le2^{-k}\bar g$ outside $B(0;\rho_k)$. Then $\bar g\sum_{k\ge0}2^{-k}q_{\rho_{k+1}}$ is a ceiling. If $\supp(g)\subset B(0;\rho)$, then $\bar g\,q_\rho$ is also a ceiling.
\item In settings (a) and (b), the function
\[
\mathcal R(x):=\E^x\Bigl[\sup_{0\le t\le\tau_{\pd\Lambda},\ t<\infty}e^{-\beta t}g(Y_t)\Bigr],\qquad x\in\overline\Lambda,
\]
is a ceiling. It is the upper bound of the dual representation of optimal stopping \cite{Rogers2002} for the zero martingale.
\end{enumerate}
In particular, a ceiling exists.
\end{lemma}

\begin{proof}
(i) Write $\tau:=\tau_{\pd\Lambda}$. By Lemma~\ref{lem:exittime}, $e(X^x_t)+t\wedge\tau$ is a martingale, so $e(X^x_t)$ is a nonnegative supermartingale. If $\phi:[0,\bar e]\to[0,\infty)$ is concave and nondecreasing, then Jensen's inequality gives, for $s\le t$,
\[
\E\bigl[\phi(e(X^x_t))\,\big|\,\mathcal F_s\bigr]\le\phi\bigl(\E[e(X^x_t)\,|\,\mathcal F_s]\bigr)\le\phi\bigl(e(X^x_s)\bigr),
\]
so $\phi(e(X^x_t))$ is a nonnegative supermartingale, and so is $e^{-\beta(t\wedge\tau)}\phi(e(X^x_t))$. Thus $\phi\circ e$ is $\beta$-superharmonic whenever it is continuous. In particular, if $g\ge0$ and $\supp(g)\subset K$, then \eqref{eq:ceile} is a ceiling: it is $\beta$-superharmonic and continuous, it vanishes on $\pd\Lambda$, and it is at least $\bar g\ge g$ on $K$ and nonnegative elsewhere; note $\min_Ke>0$ since $e$ is continuous and positive on $K$. For Brownian motion on the unit ball, $e(x)=(1-|x|^2)/d$ since $\frac12\Delta(1-|x|^2)/d=-1$, which gives \eqref{eq:w0}.

For the first ceiling in (i), $H$ is continuous on $\overline\Lambda$, by Lemmas~\ref{lem:interior} and~\ref{lem:reg}(ii), since $\Lambda$ is regular, and $\beta$-harmonic in $\Lambda$, so $e^{-\beta(t\wedge\tau)}H(X^x_t)$ is a martingale by Lemma~\ref{lem:superharm}(iii). The function $m$ is nondecreasing, and $m(t)\to m(0)=0$ as $t\downarrow0$: $(g-H)^+$ is continuous and vanishes on $\pd\Lambda$, and since $e$ is continuous and positive in $\Lambda$, the compact sets $\{e\le t\}$ decrease to $\pd\Lambda$. The constant $M$ is a concave majorant of $m$, so $\psi\le M=m(\bar e)\le\psi(\bar e)$; hence the concave function $\psi$ attains its maximum at $t=\bar e$ and is therefore nondecreasing. Given $\eta>0$, choose $\delta>0$ with $m(\delta)\le\eta$. Then $t\mapsto\eta+Mt/\delta$ is a concave majorant of $m$ (it is at least $\eta\ge m$ on $[0,\delta]$, since $m$ is nondecreasing, and at least $M\ge m$ on $[\delta,\bar e]$), so $0\le\psi(t)\le\eta+Mt/\delta$ by minimality of $\psi$. Hence $\psi(0)=0$, and $\psi$ is continuous on $[0,\bar e]$: at $0$ by the last bound, in the interior by concavity, and at $\bar e$ as we have excluded a downward jump. By the first paragraph, $\psi\circ e$ is $\beta$-superharmonic. Thus $H+\psi\circ e$ is continuous on $\overline\Lambda$, $\beta$-superharmonic in $\Lambda$ and equal to $g$ on $\pd\Lambda$, and it majorises $g$, since $H+\psi(e(x))\ge H+m(e(x))\ge H+(g-H)^+\ge g$.

(ii) Here we use Remark~\ref{rem:bmclassical} in Appendix~\ref{app:conseq}. A nonnegative superharmonic function $u$ is $\beta$-superharmonic, since $\frac12\Delta u-\beta u\le\frac12\Delta u$.

For $x\ne0$ and $a:=\sqrt{2\beta}$,
\[
\bigl(\tfrac12\Delta-\beta\bigr)e^{-a|x|}=-\frac{(d-1)a}{2|x|}\,e^{-a|x|}\le0,
\]
and $|x|^{2-d}$ is harmonic for $x\ne0$. Since $q_\rho$ is a minimum of $\beta$-superharmonic functions on $\R^d\setminus\{0\}$, and $q_\rho\equiv1$ near $0$, it is $\beta$-superharmonic; moreover it is continuous, it vanishes at infinity, and $\bar g\,q_\rho\ge g$ when $\supp(g)\subset B(0;\rho)$. For general $g$, with $\rho_k$ as in the statement, the series $\bar g\sum_{k\ge0}2^{-k}q_{\rho_{k+1}}$ converges uniformly, so its sum is continuous, $\beta$-superharmonic, and vanishes at infinity; and for $\rho_k\le|x|<\rho_{k+1}$ its $k$th term equals $2^{-k}\bar g\ge g(x)$.

(iii) Write $\tau:=\tau_{\pd\Lambda}$, which is infinite in setting (b); all suprema below are over finite times. Clearly $\mathcal R\ge g$, $|\mathcal R|\le\norm{g}$, and $\mathcal R=g$ on $\pd\Lambda$ in setting (a). In setting (b), $\mathcal R$ vanishes at infinity: given $\varepsilon>0$, choose $\rho$ with $g\le\varepsilon$ outside $B(0;\rho)$; then the supremum defining $\mathcal R(x)$ is at most $\varepsilon+\norm{g}e^{-\beta\tau_{B(0;\rho)}}\1_{\{\tau_{B(0;\rho)}<\infty\}}$, and the expectation of the last term tends to $0$ as $|x|\to\infty$ by \ref{P:trans}. For an admissible stopping time $\sigma$, the strong Markov property shows that $\E[e^{-\beta\sigma}\mathcal R(X^x_\sigma)\1_{\{\sigma<\infty\}}]$ is the expectation of $\sup_{\sigma\le t\le\tau}e^{-\beta t}g(Y_t)$ on $\{\sigma<\infty\}$, hence at most $\mathcal R(x)$. So once $\mathcal R$ is continuous, it is $\beta$-superharmonic by Lemma~\ref{lem:superharm}(i), and hence a ceiling.

For continuity, for $s>0$ let $f_s(x):=\E^x[e^{-\beta(s\wedge\tau)}\mathcal R(Y_{s\wedge\tau})]$. By the strong Markov property at $s\wedge\tau$, $f_s(x)=\E^x[N_s]$ with $N_s:=\sup_{s\wedge\tau\le t\le\tau}e^{-\beta t}g(Y_t)$. Since $N_s\ge e^{-\beta(s\wedge\tau)}g(Y_{s\wedge\tau})$, for every $\eta>0$
\begin{align*}
0\le \mathcal R(x)-f_s(x)&\le\E^x\Bigl[\omega_g\bigl(2\sup_{t\le s}|Y_t-x|\bigr)\Bigr]+\norm{g}(1-e^{-\beta s})\\
&\le\omega_g(2\eta)+2\norm{g}\delta_\eta(s)+\norm{g}(1-e^{-\beta s}),
\end{align*}
which tends to $0$ as $s\downarrow0$ and then $\eta\downarrow0$, uniformly for $x$ in bounded sets. So it suffices that each $f_s$ is continuous on $\Lambda$, and, in setting (a), that $\mathcal R(x)\to g(y)$ as $x\to y\in\pd\Lambda$.

Fix $x_0\in\Lambda$ and $\eta>0$ with $B(x_0;2\eta)\subset\Lambda$. For $x\in B(x_0;\eta)$ and $0<\epsilon<s$, the Markov property at time $\epsilon$ gives
\[
\Bigl|f_s(x)-e^{-\beta\epsilon}\,\E^x\bigl[f_{s-\epsilon}(Y_\epsilon)\bigr]\Bigr|\le2\norm{g}\,\PP^x(\tau\le\epsilon)\le2\norm{g}\,\delta_\eta(\epsilon),
\]
with $f_{s-\epsilon}$ extended by $0$ outside $\overline\Lambda$. The right-hand side tends to $0$ as $\epsilon\downarrow0$, uniformly in $x$, and $x\mapsto\E^x[f_{s-\epsilon}(Y_\epsilon)]$ is continuous by \ref{P:tv}. So $f_s$ is a uniform limit of continuous functions on $B(x_0;\eta)$, hence continuous there.

Finally, in setting (a), let $y\in\pd\Lambda$; since $\mathcal R\ge g$, it suffices to show $\limsup_{x\to y}\mathcal R(x)\le g(y)$. Fix $\eta,s>0$ and let $E:=\{\tau\le s,\ \sup_{t\le s}|Y_t-x|\le\eta\}$. On $E$, every $t\le\tau$ satisfies $t\le s$ and $|Y_t-y|\le\eta+|x-y|$, so
\[
e^{-\beta t}g(Y_t)\le g(Y_t)+\norm{g}(1-e^{-\beta s})\le g(y)+\omega_g(\eta+|x-y|)+\norm{g}(1-e^{-\beta s}).
\]
Off $E$, the maximum is at most $\norm{g}\le g(y)+2\norm{g}$. Hence
\[
\mathcal R(x)-g(y)\le\omega_g(\eta+|x-y|)+\norm{g}(1-e^{-\beta s})+2\norm{g}\bigl(\PP^x(\tau>s)+\delta_\eta(s)\bigr).
\]
By Lemma~\ref{lem:reg}(i), $\PP^x(\tau>s)\to0$ as $x\to y$ for each fixed $s$, since $y$ is regular. Letting $x\to y$, then $s\downarrow0$, then $\eta\downarrow0$, gives the claim.
\end{proof}

In setting (b), property \ref{P:trans} is also necessary for ceilings to exist for every gain satisfying Assumption~\ref{ass:breg}. If $g\ge1$ on $B(0;\rho)$ and $C$ is a ceiling, then by Lemma~\ref{lem:superharm}(ii), with $\tau:=\tau_{B(0;\rho)}$, $C(x)\ge\E^x[e^{-\beta\tau}C(Y_\tau);\,\tau<\infty]\ge\E^x[e^{-\beta\tau};\,\tau<\infty]$, and $C$ vanishes at infinity.

\begin{remark}[Regularity is necessary]\label{rem:regular}
Regularity cannot be dropped from setting (a); we give the argument for Brownian motion. Let $\Lambda$ be bounded and open with a boundary point $y$ that is not regular, as for $\Lambda=\{0<|x|<1\}$ and $y=0$, and define $X^x$, admissible stopping times and $V$ as in setting (a). By Lemma~\ref{lem:reg}(i) there are $\delta,c>0$ and points $x_n\to y$ in $\Lambda$ with $\PP^{x_n}(\tau_{\pd\Lambda}>\delta)\ge c$ for all $n$. Choose $r>0$ with $\PP(\max_{t\le\delta}|B_t|<2r)\le c/2$, and let $g(z):=\min\{|z-y|/r,1\}$. Once $|x_n-y|<r$, stopping $X^{x_n}$ when it first reaches distance $r$ from $y$, or at absorption if that comes first, gives $V(x_n)\ge e^{-\beta\delta}c/2$, whereas $V(y)=g(y)=0$. Thus $V$ is discontinuous at $y$. Every ceiling majorises $V$ by the supermartingale property and equals $g$ at $y$, so no ceiling exists, and $V$ is not a uniform limit of continuous functions.
\end{remark}

The function $\psi$ in (i) is a one-dimensional object: the least concave majorant of the scalar function $m$, which only records how fast $(g-H)^+$ decays towards $\pd\Lambda$. It is unrelated to the value function, and its construction involves no optimal stopping. In particular there is no circularity with the concave-majorant representation of one-dimensional value functions \cite{Dynkin_Yushkevich,Dayanik_Karatzas}, which the construction of this paper generalises.

\begin{remark}\label{rem:psi}
Apart from establishing its properties, the proof of Lemma~\ref{lem:ceiling}(i) uses only that $\psi:[0,\bar e]\to[0,\infty)$ is continuous, concave and nondecreasing, with $\psi(0)=0$ and $\psi\ge m$. Any such function may therefore be used in place of the least concave majorant. Continuity at $0$ must be required explicitly, since a concave function may jump there, and it is needed for the ceiling to be continuous at $\pd\Lambda$. An explicit choice, if $M>0$, is
\begin{equation}\label{eq:psi}
\psi(t):=\min\Bigl\{M,\ \inf_{k\ge1}\bigl(2^{-k}M+Mt/\delta_k\bigr)\Bigr\},
\end{equation}
where, for each $k\ge1$, $\delta_k\in(0,\bar e]$ is any number with $m(\delta_k)\le2^{-k}M$; such numbers exist since $m(t)\to0$ as $t\downarrow0$. Thus $\{e\le\delta_k\}$ is a boundary layer on which $(g-H)^+\le2^{-k}M$. The function \eqref{eq:psi} is an infimum of nondecreasing affine functions, hence concave and nondecreasing, and $\psi(0)=0$. It majorises $m$, since for $t\le\delta_k$ the $k$th affine function is at least $2^{-k}M\ge m(\delta_k)\ge m(t)$, and for $t\ge\delta_k$ it is at least $M\ge m(t)$. It is continuous: at $0$ since $\psi(t)\le2^{-k}M+Mt/\delta_k$ for every $k$, and on $(0,\bar e]$ since it is concave and nondecreasing. For example, if $g\ge0$ is supported in a compact set $K\subset\Lambda$, then $H=0$ and $m=0$ on $[0,\delta]$ for every $\delta<\min_Ke$, so $\delta_k=\delta$ is a valid choice for every $k$; then $\psi(t)=\min\{M,Mt/\delta\}$, which gives the ceiling $\bar g\min\{1,e/\delta\}$. Similarly, any continuous nonnegative $\beta$-superharmonic function $u$ which is positive in $\Lambda$ and zero on $\pd\Lambda$ may be used in place of $e$: with $\tau:=\tau_{\pd\Lambda}$, by Lemma~\ref{lem:superharm}(i) it suffices that $\E[e^{-\beta(r\wedge\tau)}\psi(u(X^x_r))]\le\psi(u(x))$ for $r\ge0$, which follows as in the proof of Lemma~\ref{lem:ceiling}(i), from $\lambda\psi(y)\le\psi(\lambda y)$ for $\lambda\in[0,1]$ (valid since $\psi$ is concave with $\psi(0)=0$), Jensen's inequality, and the monotonicity of $\psi$. On the unit ball, for Brownian motion, $1-|x|$ is such a function.
\end{remark}

\section*{Acknowledgements}

\emph{Use of AI tools.} This version was prepared with substantial assistance from AI tools. The construction of the iterated envelopes is the author's. The arguments that extend the results of the earlier version of this preprint, chiefly to more general processes and state spaces and to exponential discounting, and that remove a sufficient condition imposed there on the constructed envelopes, were developed with Claude (Anthropic), and have been checked by the author, who takes full responsibility for them. Claude was also used to edit and restructure the text, write out standard results, check proofs and suggest corrections, and prepare the \LaTeX{} source, and other large language models were used to search the literature.

\bibliographystyle{unsrtnat}
\bibliography{bibliography}
\end{document}